\documentclass[sn-mathphys]{sn-jnl}

\theoremstyle{thmstyleone}%

\usepackage{amsmath}
\usepackage{amssymb}
\usepackage{longtable}
\usepackage{epstopdf}
\usepackage{multirow,makecell}
\usepackage{longtable}
\usepackage{hyperref}
\usepackage{graphicx}
\usepackage{algorithm,algorithmicx}

\usepackage{pdflscape}
\usepackage{rotating}  % µ¼ÑÔÇøÌí¼Ó
\usepackage{graphicx}
\usepackage{subfigure}
\usepackage{enumerate}
\usepackage{xcolor}
\usepackage{color}
\usepackage{amsmath}
\usepackage{amssymb,epsfig,multirow}
\usepackage{float}
\usepackage{color}
\usepackage{booktabs}
\usepackage{enumerate}
\usepackage{algorithm}
\usepackage{algpseudocode}
\usepackage{caption}

\usepackage{epsfig,multirow}
\usepackage{xcolor}
\usepackage{diagbox}

\usepackage[utf8]{inputenc}
\usepackage{epsfig,multirow}
\usepackage{xcolor}
\usepackage{diagbox}
\newtheorem{theorem}{Theorem}%  meant for continuous numbers
\newtheorem{defi}{Definition}%[section]

\theoremstyle{thmstyletwo}%
\newtheorem{lemma}{Lemma}
\theoremstyle{thmstylethree}%
\newtheorem{cor}{Corollary}%[section]
\newtheorem{prop}{Proposition}%[section]
\newtheorem{ass}{Assumption}%[section]
\begin{document}

\title[Global optimization of {\rm (QR)}]{Global optimization of quadratic root-difference minimization under elliptic annulus constraints }


\author[1]{\fnm{Meijia} \sur{Yang}}\email{myang@ustb.edu.cn}

\author*[2]{\fnm{Yong} \sur{Xia}}\email{yxia@buaa.edu.cn}

\affil[1]{\orgdiv{School of Mathematics and Physics}, \orgname{University of Science and Technology Beijing}, \orgaddress{ \city{Beijing}, \postcode{100083}, \country{P. R. China}}}

\affil[2]{\orgdiv{LMIB of the Ministry of Education, School of Mathematical Sciences}, \orgname{Beihang University}, \orgaddress{ \city{Beijing}, \postcode{100191}, \country{P. R. China}}}


\abstract{This paper studies the nonconvex quadratic root-difference minimization under elliptic annulus constraints {\rm (QR)}. We first establish the Annulus Brickman theorem and equivalently reformulate {\rm (QR)} as a 2-dimensional convex problem {\rm (HP)} with hidden variables.
%, which is challenging to solve directly since it is difficult to determine whether a point belongs to its constraints. 
We employ the Frank-Wolfe algorithm to globally solve {\rm (HP)}. A key finding is that the solutions of the Frank-Wolfe subproblems, which are traditionally viewed as mere auxiliary updates, are proven to be $O(1/\sqrt{k})$-approximate solutions of the original problem {\rm (QR)}. This transforms an algorithmic by-product into the primary output and completely bypasses the need to solve the computationally expensive quadratic system required for solution recovery. Leveraging this recovery-free property, we develop the efficient Iterative Minimum Generalized Eigenpair (IMGE) algorithm for globally solving {\rm (QR)}. Numerical experiments confirm that IMGE converges rapidly and significantly outperforms conventional methods, especially for large-scale problems.
%validate the convergence and efficiency of our proposed algorithm.
}




\maketitle
\jyear{2021}%
\section{Introduction}
Consider the following quadratic root-difference minimization problem under elliptic annulus constraints: 
\begin{eqnarray*}
{\rm(QR)}~ &\min\limits_{x\in\Bbb R^n}& q(x)= x^TAx-\sqrt{x^TBx}\\
~~&{\rm s.t.}&x\in \mathcal{W}:=\left\{x\in\Bbb R^n:~\alpha \leq x^TCx\leq \beta\right\},
\end{eqnarray*}
where $\alpha<\beta<+\infty$, and $A,~B,~C\in\Bbb S^{n\times n}$ are all symmetric matrices. $B\succeq 0$ is assumed to make sure that the problem is well-defined. $C\succ 0$ is assumed so that the constraint is a non-degenerate elliptic annulus. The Slater's condition holds naturally for {\rm (QR)}, i.e., there exists $x\in\Bbb R^n$ such that $\alpha<x^TCx<\beta$. 

Throughout this paper, we work under the following assumptions, which are either mild or can be imposed without loss of generality:
\begin{ass}\label{ass1}
Assume that for any pair $A,~B\in\Bbb S^{n\times n}$ and sequence $\left\{\omega_i\right\}_{i=1}^\infty$, the multiplicity of any eigenvalue of the matrix $A-\omega_i B$ is 1.
\end{ass}
Assumption \ref{ass1} is generic in the sense that it holds under an arbitrarily small perturbation of the matrix $A$.
\begin{ass}\label{ass2}
Assume that $B\succ 0$ in {\rm (QR)}.
\end{ass}
Given Assumption \ref{ass1} and $B\succeq 0$, Assumption \ref{ass2} is mild under an arbitrarily small perturbation and it guarantees that the objective function of {\rm (QR)} is smooth.
\begin{ass}\label{ass3}
Assume that $\alpha>0$ in {\rm (QR)}.
\end{ass}
Since $C\succ 0$ implies $x^TCx\geq0$ holds for all $x\in\Bbb R^n$, the case $\alpha\geq 0$ is natural. The assumption that $\alpha>0$ is without loss of generality and we can refer to Appendix A for details.%Throughout this paper, we assume that the Slater's condition holds for {\rm (QR)}, i.e., there exists $x\in\Bbb R^n$ such that $\alpha<x^TCx<\beta$. 
%\textcolor{red}{so that the problem is bounded? (For more details, we will make an explanation in section 2.)}. 

%{\rm (RQ)} has many applications in optimization. Denote $u=x^TBx\geq 0$, when we take $\Phi(u)= u^2$, {\rm (RQ)} reduces to a quartic-regularized subproblem; When $B=I$ and $\Phi(u)= u^{\frac{3}{2}}$, {\rm (RQ)} reduces to the well-known Nesterov-Polyak cubic-regularized subproblem (\textcolor{red}{references}); When $\Phi(u)=-\sqrt{u}$, {\rm (RQ)} reduces to the homogeneously lifted two-sided extended trust region subproblem \cite{YWX2023}, whose vector case can be used as the relaxation of the Procrustes problem \cite{DSX15}. We name this problem by quadratic root-difference minimization problems under elliptic annulus constraints, and denote it by {\rm (QR)}. 
The motivation of {\rm (QR)} comes from the trust region subproblem:
\begin{eqnarray*}
\min\limits_{x^Tx\leq 1} x^TAx-2b^Tx=\min\limits_{x^Tx\leq 1} x^TAx-\sqrt{x^T(4bb^T)x},
\end{eqnarray*}
where $A\in\Bbb S^{n\times n}$, $b\in\Bbb R^n$, $x\in\Bbb R^n$. {\rm (QR)} was first proposed in \cite{YWX2023}, where the authors named it as the homogeneously lifted two-sided extended trust region subproblem. The authors mentioned only one application about {\rm (QR)} that the vector case of {\rm (QR)} can serve as the relaxation of the Procrustes problem \cite{DSX15}.  

 We will introduce more applications about {\rm (QR)} in optimization. When $rank(B)=1$, {\rm (QR)} reduces to the extended trust region subproblem with interval bounds \cite{PW2014,SW1995,WX2015}, where if $C=I$ and $\alpha=\beta$, it further reduces to the classical trust region subproblem.
%If $A=I$ and the interval bound constraint is removed, the essential meaning of the reduced problem is to compute the dominant eigenvector corresponding to the matrix $B$ (\textcolor{blue}{references LIUXIA...}); 
When $A\succ 0$, $\alpha$ and $\beta$ are appropriately set, {\rm (QR)} is equivalent to solve the eigenvector corresponding to the largest generalized eigenvalue of the matrix pencil $(B,A)$ \cite{LX2025siam}. We can refer to Appendix B for details.
{\rm (QR)} can also be equivalently reformulated as a special case of generalized nonhomogeneous {\rm CDT} (Celis, Dennis, and Tapia) subproblem \cite{CDT1985}. We can refer to Appendix C for details. 

{\rm (QR)} can be used for solving the penalty problem of some optimization problems. When we consider the homogeneous CDT subproblem with an equality constraint and a bilateral constraint:
\begin{eqnarray}
{\rm (HCDT)}&\min\limits_{x\in\Bbb R^n}& x^TAx\nonumber\\
~~~~~~~~~~~&{\rm s.t.}& x^TBx=1, \label{hcdt11}\\
~~~~~~~~~~~&& \alpha\leq x^TCx\leq \beta,\nonumber
\end{eqnarray} 
where $A,~B,~C,~\alpha,~\beta$ have the same definition as in {\rm (QR)}. Let $\rho>0$ be the penalty parameter with respect to the constraint (\ref{hcdt11}). Then the corresponding penalty problem is:
\[
\min\limits_{\alpha\leq x^TCx\leq \beta} x^TAx+\rho(\sqrt{x^TBx}-1)^2=\min\limits_{\alpha\leq x^TCx\leq \beta} x^T(A+\rho B)x-\sqrt{x^T(4\rho^2 B)x}+\rho,
\]
%It can be equivalently rewritten as
%\[
%\min\limits_{\alpha\leq x^TCx\leq \beta} x^T(A+\rho B)x-\sqrt{x^T(4\rho^2 B)x}+\rho,
%\]
which is equivalent to solve the problem {\rm (QR)}.
 
{\rm (QR)} can also be used as the subproblem of some optimization problems. In \cite{WX2019SIAM,Z2013COA,Z2014JCAM}, the authors studied the
problem of maximizing the sum of a Rayleigh quotient and a quadratic form over the unit sphere:
\begin{eqnarray*}
{\rm (SRQ)}~ \max\limits_{x^Tx=1} \frac{x^TD_1x}{x^TD_2x}+x^TD_3x,
\end{eqnarray*}
where $D_i\in\Bbb S^{n\times n},~i=1,2,3$ are symmetric matrices. It is assumed that $D_1\succ 0$ and $D_2\succ 0$. 
We can apply the Quadratic Transform technique, which was proposed in the communication area in \cite{SY2018IEEE} in 2018, on {\rm (SRQ)} and obtain its equivalent reformulation:
\begin{eqnarray*}
{\rm (ESRQ)}~\min_{y\in\Bbb R}\left\{g(y)= \min\limits_{x^Tx=1} y^2x^TD_2x-x^TD_3x-2y\sqrt{x^TD_1x}\right\}
\end{eqnarray*} 
in the sense that $v{\rm (SRQ)}=v{\rm (ESRQ)}$. It is worth mentioning that the Quadratic Transform technique was originally proposed in \cite{KK1990AOR} in 1990 on the sum of linear fractional programming problems. We can incorporate the quadratic fit line search method  \cite{NSX2016} to find the optimal $y$ in {\rm (ESRQ)}. Then evaluating $g(y)$ is equivalent to solve the problem {\rm (QR)}.
%where when $y$ is fixed, it reduces to a special case of {\rm (QR)}. 

{\rm (QR)} is a nonconvex programming problem since both the objective function and the constraints are nonconvex. The hidden convexity property of {\rm (QR)} was first revealed in \cite{YWX2023} with non-quadratic extension of homogeneous S-lemma, where the authors equivalently reformulate {\rm (QR)} as a linear conic programming problem {\rm (LC)}:
%{\rm (QR)} is equivalent to a linear conic programming problem:
\begin{eqnarray*}
{\rm (LC)}&\sup\limits_{\lambda\in \Bbb R^3_+,\mu\geq0}& \lambda_1\alpha-\lambda_2\beta-\mu\\
&{\rm s.t.}& A+(\lambda_2-\lambda_1)C-\lambda_3 B  \succeq 0,\\
&&\left\|\left(\begin{array}{c}1\\ \lambda_3-\mu \end{array}\right)
\right\|\le \lambda_3+\mu
\end{eqnarray*}
%
%\begin{eqnarray*}
%&{\rm (LC)}&\sup_{\lambda\in \Bbb R^3_+,\mu\geq0} \left\{\lambda_1\alpha-\lambda_2\beta-\mu: A+(\lambda_2-\lambda_1)C-\lambda_3 B  \succeq 0,\right.\nonumber\\
%&&~~~~~~~~~~~~~~~~~~~~~~~~~~~~~~~~~~~\left. ~\left\|\left(\begin{array}{c}1\\ \lambda_3-\mu \end{array}\right)
%\right\|\le \lambda_3+\mu\right\}
%\end{eqnarray*}
in the sense that $v({\rm QR})=v({\rm LC})$. Thus we can solve {\rm (QR)} through {\rm (LC)} with CVX and recover its optimal solution in polynomial time \cite{YZ2003}. However, for large-scale implementations of {\rm (LC)}, it may exhibit certain limitations in computational efficiency \cite{BoydBook}. 
%For larger-scale implementations of the (LC) model, the computational efficiency during the solving process tends to degrade.

%\textcolor{red}{need to be updated: emphasis some innovation?}
In this paper, we study the feasibility of developing an efficient algorithm for solving large-scale {\rm (QR)} while maintaining theoretical guarantees in convergence and numerical stability. The contributions are stated below.

1. We develop the Annulus Brickman theorem to equivalently reformulate {\rm (QR)} as a convex programming problem {\rm (HP)} with hidden variables. We observe that there are three quadratic forms in {\rm (QR)}, which lead us to connect it with the well known results proposed by Brickman \cite{B1961} and Calabi-Polyak \cite{C1982,P2007,P1998}. 
%Based on an extension of Hausdorff's result \cite{H1919} regarding the convexity of image of homogeneous quadratic mapping, 
Brickman proved the convexity of the image of two homogeneous quadratic functions over the unit sphere if $n\geq 3$. Calabi and Polyak separately proved that if $n\geq 3$ and there exists a positive definite linear combination of three quadratic forms, then their image set is convex. Following these results, we successfully develop the Annulus Brickman theorem and reformulate {\rm (QR)} as a 2-dimensional convex programming problem {\rm (HP)} with respect to hidden variables. However, since it is not easy to determine whether a point belongs to the constraint of {\rm (HP)}, the ellipsoidal algorithm or the projected gradient method is not applicable \cite{B1999,N2004}. As a result, it is challenging to solve the equivalent convex programming problem {\rm (HP)} directly.

2. We apply the Frank-Wolfe algorithm to solve the 2-dimensional convex programming problem {\rm (HP)} with respect to the hidden variables and prove that the algorithm converges with respect to not only the objective value but also the iterates. The well known Frank-Wolfe algorithm is useful for smooth convex optimization and it has the $O(1/k)$ convergence rate \cite{P2024}. The Frank-Wolfe method was originally introduced in \cite{FW1956} for the minimization of quadratic convex objective function over the polytopes. It was then extended to many different optimization problems and intensively studied \cite{BRZ2024,FG2016}. In this paper, we apply the Frank-Wolfe method on the hidden variables to globally solve {\rm (HP)} and prove the convergence result. The remaining difficulty is to recover the optimal solution of {\rm (QR)} from that of {\rm (HP)} through solving a quadratic system, which is known to be computationally difficult \cite{NLSX2026}. 

3. We develop an iterative minimum generalized eigenpair (IMGE) algorithm to solve {\rm (QR)} and prove that the algorithm converges with respect to the objective value. We prove that the iterates generated by the Frank-Wolfe subproblem are $\delta_k$-approximate solutions of {\rm (QR)} with $\delta_k=O(1/\sqrt{k})$. Thus, it eliminates the need to solve a quadratic system and greatly improves the computational efficiency. Consequently, we develop the IMGE algorithm to solve {\rm (QR)}, where the main computation is to solve the minimum generalized eigenpair of a matrix pencil. We provide some numerical results to demonstrate the convergency and computational efficiency.

Throughout this paper, $\mathcal{S}^n$ denotes the set of symmetric matrices in $\Bbb R^n$. {\rm $v(\cdot)$} denotes the optimal value of {\rm $(\cdot)$}. 
%For a vector $x=(x_1,...,x_n)$, $x\in\Bbb R_+^n$ means that $x_i\geq 0,~i=1,...,n$. 
$I_n$ denotes an $n\times n$ identity matrix. $A\in\Bbb {S}^{n\times n}$ indicates that $A$ is a symmetric matrix. For $A\in\Bbb S^{n\times n}$, $A\succ (\succeq)$ denotes that $A$ is positive (semi-)definite. 
For the matrices $A\in\Bbb S^{n\times n}$ and $B(\in\Bbb S^{n\times n})\succ 0$, $\lambda_{\max}(A,B)$ and $\lambda_{\min}(A,B)$ denote the largest and smallest generalized eigenvalue of the 
matrix pencil $(A,B)$, respectively. 
%$B\in\Bbb R^{n\times n}$, $A\bullet B$ denotes the trace of $A\times B$. 

The paper is organized as follows.  
In section 2, we equivalently reformulate {\rm (QR)} as a 2-dimensional convex programming problem {\rm (HP)} with hidden variables. In section 3, we apply the Frank-Wolfe algorithm to globally solve {\rm (HP)} and prove that the algorithm converges with respcet to the objective value. In section 4, we further prove that the Frank-Wolfe algorithm converges with respect to the iterates. 
In section 5, we develop an iterative minimum generalized eigenpair {\rm (IMGE)} algorithm to globally solve {\rm (QR)} and prove that it converges with respect to the objective value. In section 6, numerical results are carried out and demonstrate the convergency and the computational efficiency of IMGE algorithm compared with the method of solving {\rm (QR)} through {\rm (LC)} with CVX. Conclusions are made in section 7.

\section{Challenging convex reformulation of {\rm (QR)}}
%\begin{eqnarray*}
%{\rm(RQ)}~ &\min\limits_{x\in\Bbb R^n}& q(x)= x^TAx+\Phi(x^TBx)\\
%~~&{\rm s.t.}&x\in W:=\left\{x\in\Bbb R^n:~\alpha \leq x^TCx\leq \beta\right\},
%\end{eqnarray*}
In this section, we equivalently reformulate {\rm (QR)} as a convex programming problem (different from {\rm (LC)}). However, it is still challenging to solve the equivalent convex problem directly.

Letting $s=x^TAx$, $t=x^TBx$, {\rm (QR)} is equivalent to the following 2-dimensional problem with respect to hidden variables $(s,t)$:
\begin{eqnarray}
{\rm (HP)}~&\min\limits_{(s,t)\in\Bbb R^2}& f(s,t)=s-\sqrt{t}\nonumber\\
&{\rm s.t.}& (s,t)\in \Omega=\left\{(x^TAx,x^TBx): \alpha\leq x^TCx\leq\beta,x\in\Bbb R^n\right\}.\nonumber%\label{c1}
%&&\alpha\leq w\leq \beta.\label{c2}
\end{eqnarray} 
Since $C\succ 0$, $\alpha >0$, according to $0<\alpha\leq x^TCx\leq \beta$, we can conclude that $x\neq 0$. Furthermore, since $B\succ 0,~C\succ 0$, we can conclude that $t>0$.

%Denote $S_1=\left\{(s,t,w):(s,t,w)\in S,~\alpha\leq w\leq \beta\right\}$.
We first review the well-known result proposed by Brickman, which reveals the convexity of the image of two homogeneous quadratic functions over the unit sphere:
\begin{theorem}[Brickman, \cite{B1961}]\label{Brickman}
Let $h_1,~h_2:\Bbb R^n\rightarrow \Bbb R$ be homogeneous quadratic functions. If $n\geq 3$, then the set
$
\mathcal{W}_{1}=\left\{(h_1(x),h_2(x)):x\in\Bbb R^n,\|x\|=1\right\}\subset \Bbb R^2
$ is convex.
\end{theorem}
Based on Theorem \ref{Brickman}, Calabi and Polyak separately proved the following famous theorem: 
%\begin{theorem}[Calabi-Polyak, \cite{C1982,P1998}]\label{Polyak}
%If $n\geq 3$ and $h_1,h_2,h_3:\Bbb R^n\rightarrow \Bbb R$ are homogeneous quadratic functions such that there exists a positive definite linear combination of them, then the set $W_2=\left\{(h_1(x),h_2(x),h_3(x)):x\in\Bbb R^n\right\}\subset \Bbb R^3
%$
%is convex.
%\end{theorem}
\begin{theorem}[Calabi-Polyak, \cite{C1982,P1998}]\label{Polyak}
Let $n\geq 3$ and $\mathcal{G}=\left\{G_1,G_2,G_3\right\}\subset \mathcal{S}^n$. Suppose that $\mathcal{G}$ has a positive definite linear combination, then the joint numerical range $\left\{(x^TG_1x,x^TG_2x,x^TG_3x):x\in\Bbb R^n\right\}$ is an acute closed convex cone.
\end{theorem}

Now we are ready to propose the following Annulus Brickman theorem, which is essential in the following analysis:
\begin{theorem}[Annulus Brickman Theorem]\label{ab}
Let $A_1\in\Bbb S^{n\times n}$, $A_2\in\Bbb S^{n\times n}$ and $0< \alpha_1<\alpha_2<+\infty$. If $n\geq 3$, then the set 
\[
\mathcal{W}_3=\left\{(x^TA_1x,x^TA_2x):\alpha_1\leq x^Tx\leq\alpha_2,~x\in\Bbb R^n\right\}\subset \Bbb R^2
\]
\end{theorem}
is convex and compact.
\begin{proof}
Consider $A_1,A_2,I_n\in\Bbb S^{n\times n}$, these three matrices admit a positive definite linear combination (e.g., $0A_1+0A_2+I_n\succ 0$). According to Calabi-Polyak theorem (Theorem \ref{Polyak}), the set
\[
\mathcal{T}_1=\left\{(x^TA_1x,x^TA_2x,x^Tx):x\in\Bbb R^n\right\}\subset \Bbb R^3
\]
 is convex. Moreover, the following set
\[
\mathcal{T}_2=\left\{(z_1,z_2,z_3): \alpha_1\leq z_3\leq \alpha_2\right\}\subset\Bbb R^3
\]
is convex as it is an intersection of two closed half-spaces. Consequently, $\mathcal{T}_3\triangleq \mathcal{T}_1\cap \mathcal{T}_2$ is convex. 

Define  
$
\mathcal{T}_4=\left\{x\in\Bbb R^n: \alpha_1\leq x^Tx\leq \alpha_2\right\},
$
which is closed and bounded, hence compact in $\Bbb R^n$. The mapping 
\[
\Psi: \Bbb R^n\rightarrow \Bbb R^3,~\Psi(x)=(x^TA_1x,x^TA_2x,x^Tx)
\]
is continuous, and $\mathcal{T}_3=\Psi(\mathcal{T}_4)$. Therefore, $\mathcal{T}_3$ is compact as the continuous image of a compact set.

Let $\Phi:\Bbb R^3\rightarrow \Bbb R^2$ be the linear projection such that 
\[\Phi(\left\{(c_1,c_2,c_3):c_1,c_2,c_3\in\Bbb R\right\})=\left\{(c_1,c_2):c_1,c_2\in\Bbb R\right\}.\] 
Then $\mathcal{W}_3=\Phi(\mathcal{T}_3)$. Both the convexity and compactness are preserved under continuous maps. Since $\Phi$ is linear and thus continuous, it follows that $\mathcal{W}_3$ is convex and compact. The proof is complete.
\end{proof}

\begin{cor}\label{cor1}
{\rm (HP)} is a convex programming problem.
\end{cor}
\begin{proof}
It is obvious that the objective function $f(s,t)$ defined in {\rm (HP)} is convex. Since $C\succ 0$, we can make a Cholesky decomposition of $C$ that $C=UU^T$, where $U\in\Bbb R^{n\times n}$ is a lower triangular and invertible matrix. Let $y=U^Tx$, then the set $\Omega$ defined in {\rm (HP)} can be equivalently reformulated as:
\[
\Omega=\left\{(y^T\tilde{A}y,y^T\tilde{B}y):\alpha\leq y^Ty\leq \beta, y\in\Bbb R^n\right\},
\]
where $\tilde{A}=U^{-1}AU^{-T}$ and $\tilde{B}=U^{-1}BU^{-T}$ are symmetric matrices. Then by the Annulus Brickman theorem (Theorem \ref{ab}), $\Omega$ is a convex set. Consequently, {\rm (HP)} is a convex programming problem. The proof is complete.
%
%
%The constraint can be equivalently rewritten as 
%\[
%(s,t,w)\in S=\left\{(s,t,w)\in S_1,~\alpha\leq w\leq \beta\right\},\]
%where $S_1=\left\{(x^TAx,x^TBx,x^TCx): x\in\Bbb R^n\right\}$. According to Theorem \ref{Polyak}, $S_1$ is a convex set. Besides, $\alpha\leq w \leq \beta$ is a linear constraint. Hence the constraint of {\rm (P)} is convex. 
%Thus{\rm (P)} is a convex programming problem.
\end{proof}

% There are many well-known algorithms for 

Though {\rm (HP)} is a 2-dimensional convex programming problem, the ellipsoidal algorithm or the projected gradient method is not applicable. Since the feasible set $\Omega$ lacks an explicit analytical description, it is even not easy to determine whether a point belongs to $\Omega$. We can solve {\rm (QR)} in two stages. First, we obtain the optimal solution $(s^*,t^*)$ of the 2-dimensional convex programming problem {\rm (HP)}. Then, we recover the optimal solution $x^*$ of {\rm (QR)} by solving the quadratic system 
\begin{eqnarray}\label{qua}
x^TAx=s^*,~x^TBx=t^*
\end{eqnarray}
subject to $\alpha\leq x^TCx\leq \beta$, which is computationally difficult \cite{NLSX2026}.

%We can equivalently view {\rm (P)} as a bilevel programming problem:
%\textcolor{red}{
%\begin{eqnarray*}
%&\min\limits_{(s,t,x)\in\Bbb R^{n+2}}&~s-\sqrt{t}\\
%&{\rm s.t.}&~x=\arg\min\limits_{x\in \Omega_1} 0,
%\end{eqnarray*}
%}
%where $\Omega_1=\left\{x^TAx=s,x^TBx=t,\alpha\leq x^TCx\leq \beta\right\}$. In the upper level, we obtain the optimal $(s^*,t^*)$ through solving a 2-dimensional convex programming problem. However, the constraint is difficult to characterize. In the lower level, we need to resolve $x^*$ with obtained $(s^*,t^*)$, where we need to \textcolor{red}{solve the quadratic system $\Omega_1$.}

%
%\textcolor{red}{When solving {\rm (P)} at the $k$-th iteration and obtain its optimal solution, denoted by $(s_k,t_k)$, we need to solve the following quadratic constrained quadratic programming problem to obtain the optimal $x_k$:
%\[\min_{x\in\Bbb R^n}\left\{ s_k-\sqrt{t_k}:~x^TCx\geq \alpha,x^TCx\leq \beta\right\},\]
%which can be solved by the method proposed in \cite{YZ2003}.
%} 
%  
\section{Frank-Wolfe solvability of {\rm (HP)}}
%\begin{eqnarray}
%{\rm (P)}~&\min\limits_{s,t\in\Bbb R}& f(s,t)=s+\Phi(t)\nonumber\\
%&{\rm s.t.}& (s,t)\in \Omega=\left\{(x^TAx,x^TBx): \alpha\leq x^TCx\leq\beta,x\in\Bbb R^n\right\}.\nonumber%\label{c1}
%%&&\alpha\leq w\leq \beta.\label{c2}
%\end{eqnarray} 
In this section, we will apply the Frank-Wolfe algorithm on the hidden variables $(s,t)$ to solve {\rm (HP)} and prove that the algorithm converges with respect to the objective value.

Recall that $t> 0$, the gradient of $f(s,t)$ is
\[
\nabla f(s,t)=\left(1,~-\frac{1}{2\sqrt{t}}\right)^T.
\]
%where $\epsilon>0$ is a sufficiently small constant, which is added in case that $t=0$ and $\nabla f(s,t)$ is meaningless. 
Then we can apply the Frank-Wolfe framework on $(s,t)$ to solve {\rm (HP)}. At the $k$-th iteration, we need to solve the following linearized subproblem: 
\begin{equation}
\min\limits_{(s,t)\in \Omega}\nabla f(s_k,t_k)^T\left(\begin{array}{c}
s\\
t
\end{array}\right)=s-\frac{1}{2\sqrt{t_k}}t,\label{p4}%-s_k+\frac{\sqrt{t_k}}{2}
\end{equation}
whose optimal solution is denoted by $(\hat{s}_k,\hat{t}_k)$ throughout this paper. Then the search direction at the $k$-th iteration is 
\[
d_k=\left(\begin{array}{c}
\hat{s}_k-s_k\\
\hat{t}_k-t_k
\end{array}\right).
\]
Now we are facing the problem about solving (\ref{p4}). We will show in the following proposition that (\ref{p4}) becomes tractable when it is equivalently reformulated as a new problem with respect to the variable $x$ of the original problem {\rm (QR)}.
\begin{prop}\label{prop1}
 Define $M_k(\in\Bbb S^{n\times n})=A-\frac{1}{2\sqrt{t_k}}B$. Let 
 $(\lambda_g,v_g)\in(\Bbb R\times \Bbb R^n)$ be the minimum normalized generalized eigenpair of $(M_k,C)$ with $v_g^TCv_g=1$. Then the optimal solution of problem (\ref{p4}) is $(\hat{s}_k, \hat{t}_k)=(\hat{x}_k^TA\hat{x}_k,\hat{x}_k^TB\hat{x}_k)$, where $\hat{x}_k$ is given by:
 
%  % \begin{equation}\label{gcdt3}
%x_k=\left\{
%\begin{aligned}
%(\alpha v_g^TAv_g,\alpha v_g^TBv_g),&~~~if \lambda_g= 0,\\
%(\beta v_g^TAv_g,\beta v_g^TBv_g),&~~~if \lambda_g> 0,\\
%(\beta v_g^TAv_g,\beta v_g^TBv_g),&~~~if \lambda_g< 0,
%\end{aligned}
%\right.
% \end{equation}  
%  $x_k$ is given by:
% 
Case 1. If $\lambda_g=0$,
 %(equivalently, $\lambda_{min}(A,B)=\frac{1}{2\sqrt{t_k}}$; or equivalently, $M_k\succeq 0$ and $dim(ker(M_k))\geq 1$)
 then $\hat{x}_k=\sum_{i=1}^mc_iu_i$ with $m=\dim(\ker(M_k))$, $\left\{u_1,u_2,...,u_m\right\}$ being an orthonormal basis of $\ker(M_k)$, and $c_i\in\Bbb R,~i=1,...,m$ chosen such that $\alpha<\hat{x}_k^TC\hat{x}_k<\beta$. In this case, $v{\rm (GP_k)}=0$.

Case 2. If $\lambda_g>0$, then $\hat{x}_k=\sqrt{\alpha}v_g$, $v{\rm (GP_k)}=\lambda_g\alpha$.

Case 3. If $\lambda_g<0$, then $\hat{x}_k=\sqrt{\beta}v_g$, $v{\rm (GP_k)}=\lambda_g\beta$.

\end{prop}
\begin{proof}
The problem (\ref{p4}) is equivalent to the following minimization problem:
%\begin{eqnarray*}\label{eigg}
%(\text{GP}_k)&\min\limits_{x\in\Bbb R^n}&~x^TAx-\frac{1}{2\sqrt{t_k}}x^TBx,\\
%&s.t.&~\alpha\leq x^TCx\leq \beta.
%\end{eqnarray*}
\begin{eqnarray*}\label{eigg}
(\text{GP}_k)&\min\limits_{x\in\Bbb R^n}&~x^TM_kx\\
&s.t.&~\alpha\leq x^TCx\leq \beta,
\end{eqnarray*}
where $M_k=A-\frac{1}{2\sqrt{t_k}}B$. Denote the normalized minimum generalized eigenpair of the pencil $(M_k,C)$ by $(\lambda_g,v_g)\in(\Bbb R\times \Bbb R^n)$ with $v_g^TCv_g=1$. Let $\mu_1\geq 0$ and $\mu_2\geq 0$ be the Largangian multipliers corresponding to the constraints $\alpha\leq x^TCx$ and $x^TCx\leq \beta$, respectively. The Largangian function of $(\text{GP}_k)$ is
\[
L(x,\mu_1,\mu_2)=x^TM_kx+\mu_1(\alpha-x^TCx)+\mu_2(x^TCx-\beta).
\]
Let $\hat{x}_k$ be an optimal solution of $(\text{GP}_k)$. The KKT system is 
\begin{eqnarray}
&&M_k\hat{x}_k=(\mu_1-\mu_2)C\hat{x}_k,\label{kkt0}\\
&&\alpha\leq \hat{x}_k^TC\hat{x}_k\leq \beta,\nonumber\\
&&\mu_1(\alpha-\hat{x}_k^TC\hat{x}_k)=0,\label{kkt1}\\
&&\mu_2(\hat{x}_k^TC\hat{x}_k-\beta)=0,\label{kkt2}\\
&&\mu_1\geq 0,~\mu_2\geq 0.\nonumber
\end{eqnarray}
%\begin{equation*}
%\left\{
%\begin{aligned}
%M_kx=(\mu_1-\mu_2)Cx,\\
%\alpha\leq x^TCx\leq \beta,\\
%\mu_1(\alpha-x^TCx)=0,\\
%\mu_2(x^TCx-\beta)=0,\\
%\mu_1\geq 0,~\mu_2\geq 0.
%\end{aligned}
%\right.
%\end{equation*} 

 Let's consider the following three cases based on the complementary slackness conditions: 

Case (i): $\alpha<\hat{x}_k^TC\hat{x}_k<\beta$. According to (\ref{kkt1}) and (\ref{kkt2}), it holds that $\mu_1=\mu_2=0$. The equality (\ref{kkt0}) yields $M_k\hat{x}_k=0$. Thus, $\hat{x}_k\in \ker(M_k)~\symbol{92}\{0\}$ and the optimal value is $\hat{x}_k^TM_k\hat{x}_k=0$. 
This case occurs only if there exists a nonzero $x\in \ker(M_k)$ satisfying $\alpha<x^TCx<\beta$. 

Since $\hat{x}_k^TM_k\hat{x}_k=0$ is the global minimum objective value of $(\text{GP}_k)$, it holds that $M_k\succeq 0$. Otherwise a feasible point with negative objective value would exist. Because $M_k\succeq 0$ and $C\succ 0$, all the generalized eigenvalues of the pencil $(M_k,C)$ are nonnegative. Thus, $\lambda_g\geq 0$. Moreover, $M_k\hat{x}_k=0$ implies that 0 is indeed a generalized eigenvalue of $(M_k,C)$. Consequently, $\lambda_g=0$ and $\dim(\ker(M_k))\geq 1$. 

Denote $m=\dim(\ker(M_k))$, it must hold that $m\geq 1$. Let ${u_1,...,u_m}$ be an orthonormal basis of $\ker(M_k)$. Then $\hat{x}_k=\sum_{i=1}^mc_iu_i$, where $c_i,~i=1,...,m$ is chosen to satisfy $\alpha<\hat{x}_k^TC\hat{x}_k<\beta$. 
%
%\textcolor{blue}{We will further derive another equivalent description of $\lambda_g=0$.} The equality $M_k\hat{x}_k=0$ is equivalent to $A\hat{x}_k=\frac{1}{2\sqrt{t_k}}B\hat{x}_k$, which shows that $\frac{1}{2\sqrt{t_k}}$ is a generalized eigenvalue of the pencil $(A,B)$ with eigenvector $\hat{x}_k$. Furthermore, $\frac{1}{2\sqrt{t_k}}$ is indeed the minimum generalized eigenvalue of $(A,B)$. 
%Assume that there exist $\tilde{\lambda}<\frac{1}{2\sqrt{t_k}}$ and $y\neq 0$ such that $Ay=\tilde{\lambda} By$. Consider 
%\[
%M_ky=Ay-\frac{1}{2\sqrt{t_k}}By=(\tilde{\lambda}-\frac{1}{2\sqrt{t_k}})By,
%\] 
%which indicates that $\tilde{\lambda}-\frac{1}{2\sqrt{t_k}}$ is a generalized eigenvalue of the pencil $(M_k,B)$ with eigenvector $y$. Since $M_k\succeq 0$ and $B\succ 0$, it holds that $\tilde{\lambda}-\frac{1}{2\sqrt{t_k}}\geq 0$. Therefore, there does not exist a generalized eigenvalue of the pencil $(A,B)$ that is smaller that 
%$\frac{1}{2\sqrt{t_k}}$. Thus, $\frac{1}{2\sqrt{t_k}}=\lambda_{min}(A,B)$.

Case (ii): $\hat{x}_k^TC\hat{x}_k=\alpha$. According to (\ref{kkt1}), it holds that $\mu_2=0$. Then (\ref{kkt0}) turns to be $M_k\hat{x}_k=\mu_1C\hat{x}_k$. 
%Let $(\lambda_\alpha(\in\Bbb R),v_\alpha(\in\Bbb R^n))$ be the normalized minimum generalized eigenpair of $(M_k,C)$ with $v_\alpha^TCv_\alpha=1$. 
If $\lambda_g> 0$, we can set $\mu_1=\lambda_g\geq 0$ which satisfies the nonnegativity requirement. Then the optimal solution is $\hat{x}_k=\sqrt{\alpha}v_g$, with an optimal value $\hat{x}_k^TM_k\hat{x}_k=\mu_1\hat{x}_k^TC\hat{x}_k=\lambda_g\alpha$. 

Case (iii): $\hat{x}_k^TC\hat{x}_k=\beta$. According to (\ref{kkt2}), it holds that $\mu_1=0$. Then (\ref{kkt0}) turns to be $M_k\hat{x}_k=-\mu_2C\hat{x}_k$. 
%Let $(\lambda_\beta(\in\Bbb R),v_\beta(\in\Bbb R^n))$ be the normalized minimum generalized eigenpair of $(M_k,C)$ with $v_\beta^TCv_\beta=1$. 
If $\lambda_g< 0$, we can set $\mu_2=-\lambda_g$ which satisfies the nonnegativity requirement. Then the optimal solution is $\hat{x}_k=\sqrt{\beta}v_g$, with an optimal value $\hat{x}_k^TM_k\hat{x}_k=\mu_2\hat{x}_k^TC\hat{x}_k=\lambda_g\beta$. 

It is worth noting that Cases (ii) and (iii) remain valid when $\lambda_g=0$. We treat $\lambda_g=0$ as a distinct case to provide a more comprehensive theoretical framework. Besides, Cases (i)-(iii) correspond to Cases 1-3 in the proposition.
%
%In summary, we can solve $(\text{GP}_k)$ in the following way. First of all, we can solve the normalized minimum generalized eigenpair $(\lambda_g(\in\Bbb R),v_g(\in\Bbb R^b))$ of the pencil $(M_k,C)$. Then, we obtain the following three cases based on $\lambda_g$:
%
%(i) If $\lambda_g=0$, then $x_k=\bar{x}$, $v{\rm (GP_k)}=0$;
%
%(ii) If $\lambda_g\geq 0$, then $x_k=\sqrt{\alpha}v_g$, $v{\rm (GP_k)}=x_k^TM_kx_k=\mu_1x_k^TCx_k=\lambda_g\alpha$;
%
%(iii) If $\lambda_g<0$, then $x_k=\sqrt{\beta}v_g$, $v{\rm (GP_k)}=x_k^TM_kx_k=\mu_2x_k^TCx_k=\lambda_g\beta$.\\  

Finally, according to the constraint in the problem (\ref{p4}) that $(s,t)\in \Omega$, the optimal solution is $(\hat{s}_k,\hat{t}_k)=(\hat{x}_k^TA\hat{x}_k,\hat{x}_k^TB\hat{x}_k)$. 
The proof is complete. 
\end{proof}

Let $(s^*, t^*)$ be the optimal solution of {\rm (HP)}, $(s_{k},t_{k})$ be the iterates of the Frank-Wolfe algorithm, $(\hat{s}_{k},\hat{t}_{k})$ be the optimization of (\ref{p4}), and $\hat{x}_k$ be the optimal solution of $(\text{GP}_k)$ throughout this paper. We summarize the framework of the Frank-Wolfe algorithm for solving {\rm (HP)} as bellow:
\begin{algorithm}[H]\caption{Frank-Wolfe Algorithm for Globally Solving {\rm (HP)}.}\label{algfw}
\textbf{Input:} Initial point $(s_1,t_1)\in\Omega$, $k=1$.\\%, terminate tolerance $\Delta$
\textbf{Output:} An approximate optimal solution of {\rm (HP)}.
\begin{algorithmic}[1]
        \While{termination criterion not satisfied}
        \State Solve the optimal solution $(\hat{s}_k,\hat{t}_k)$ of problem (\ref{p4}) via Proposition \ref{prop1}.
        \State Update stepsize $\gamma_k$.    
        \State 
        Update
        $(s_{k+1},t_{k+1})=(1-\gamma_k)(s_k,t_k)+\gamma_k(\hat{s}_k,\hat{t}_k)$. 
        \State 
        Update $k:=k+1$.
        \EndWhile
    \end{algorithmic}  
    \end{algorithm}
\begin{defi}{\rm(L-smooth. \cite{P2024})}\label{def1}
Let $\mathcal{P}$ be a convex set and $g:\mathcal{P}\rightarrow \Bbb R$ be a differentiable function. Then $g$ is L-smooth if it holds for all $x,y\in \mathcal{P}$ that
\[g(y)-g(x)\leq\left <\nabla g(x),y-x\right >+\frac{L}{2}\|y-x\|^2.\]
\end{defi}
\begin{defi}{\rm($\epsilon$-approximate solution)}
For $\epsilon>0$, $x_k$ is an $\epsilon$-approximate solution of problem {\rm (P)} if 
\[
v{\rm (P)}\leq g(x_k)\leq v{\rm (P)}+\epsilon
\]
with $g(x)$ being the objective function of {\rm (P)}.
\end{defi}
\begin{theorem}{\rm (Primal convergence of the Frank-Wolfe algorithm. \cite{J2013,P2024})}\label{fw1}
Let $h$ be an L-smooth convex function and let $\mathcal{C}$ be a compact convex set of diameter $D$. Consider the iterates of the Frank-Wolfe algorithm for solving $\min\limits_{x\in\mathcal{C}}h(x)$ with stepsize determined by either $\gamma=\frac{2}{k+2}$ or the exact line search. Then the following holds
\[
h(x_k)-h(x^*)\leq \frac{2LD^2}{k+2},
\]
and hence for any accuracy $\epsilon>0$, we have $h(x_k)-h(x^*)\leq \epsilon$ for all $k\geq \frac{2LD^2}{\epsilon}$.
\end{theorem}
The convergency of Algorithm \ref{algfw} is ensured by the following theorem:
\begin{theorem}{\rm(Primal convergence of Algorithm \ref{algfw})}\label{thm3}
If we apply Algorithm \ref{algfw} to solve {\rm (HP)} with stepsize $\gamma_k$ determined by either $\gamma_k=\frac{2}{k+2}$ or the exact line search, then it holds for every $k\geq 1$ that:
\begin{eqnarray}
f(s_k,t_k)-f(s^*,t^*)\leq \frac{2LD^2}{k+2},\label{converge}
%f(x_k)-f(x^*)\leq \frac{2C_f}{k+2}(1+\delta),
\end{eqnarray}
%for the iterates $x_k$, 
where $(s_k,t_k)\in\Omega$ is the $k$-th iterate, $(s^*,t^*)\in \Omega$ is an optimal solution of {\rm (HP)}. Here
\begin{eqnarray}
L&=&\frac{1}{4(\lambda_{\min}(B,C)\alpha)^{3/2}},\label{eql}\\
D&=&\sqrt{(s_{\max}-s_{\min})^2+(\lambda_{\max}(B,C)\beta-\lambda_{\min}(B,C)\alpha)^2}\label{eqd}
\end{eqnarray}
with 
\begin{eqnarray}
s_{\max}&=&\max\{\lambda_{\max}(A,C)\beta,~\lambda_{\max}(A,C)\alpha\},\label{eq5}\\
s_{\min}&=&\min\{\lambda_{\min}(A,C)\alpha,~\lambda_{\min}(A,C)\beta\}.\label{eq6}
%&&t_{\max}=\max\{\lambda_{\max}(B,C)\beta,\lambda_{\max}(B,C)\alpha\}, \label{eq7}\\
%&&t_{\min}=\min\{\lambda_{\min}(B,C)\alpha,\lambda_{\min}(B,C)\beta\}.\label{eq8}
\end{eqnarray}
%Consequently, for any accuracy $\epsilon>0$, $q(x_k)-q(x^*)\leq\epsilon$ holds for all $k\geq\frac{2LD^2}{\epsilon}$.
\end{theorem}
\begin{proof}
The Algorithm \ref{algfw} for solving {\rm (HP)} inherits directly from the framework of the Frank-Wolfe algorithm, whose convergence analysis is standard \cite{P2024}. According to Theorem \ref{ab} and Corollary \ref{cor1}, $f(s,t)$ is a convex function and the set $\Omega$ defined in {\rm (HP)} is convex and compact. Therefore, according to Theorem \ref{fw1}, it suffices to determine an upper bound $D$ of the diameter of $\Omega$ and the smoothness constant $L$.

The diameter of $\Omega$ is:
\begin{eqnarray*}
D_{\Omega}&=&\max_{(s_1,t_1),(s_2,t_2)\in \Omega}\sqrt{(s_1-s_2)^2+(t_1-t_2)^2}\\
&\leq&\sqrt{(s_{\max}-s_{\min})^2+(t_{\max}-t_{\min})^2}\\
&=&(\ref{eqd})=D,
\end{eqnarray*}
where it holds in the inequality that
\begin{eqnarray*}
s_{\max}&=&\max_{\alpha\leq x^TCx\leq\beta}x^TAx=\left\{
\begin{aligned}
\lambda_{\max}(A,C)\beta,&~~~if~\lambda_{\max}(A,C)\geq 0,\\
\lambda_{\max}(A,C)\alpha,&~~~if~\lambda_{\max}(A,C)< 0,
\end{aligned}
\right.\label{eq1}\\
s_{\min}&=&\min_{\alpha\leq x^TCx\leq\beta}x^TAx=\left\{
\begin{aligned}
\lambda_{\min}(A,C)\alpha,&~~~if~\lambda_{\min}(A,C)\geq 0,\\
\lambda_{\min}(A,C)\beta,&~~~if~\lambda_{\min}(A,C)< 0.
\end{aligned}
\right.\label{eq2}
\end{eqnarray*}
Thus (\ref{eq5})-(\ref{eq6}) hold. Since $B\succ 0$, $C\succ 0$ and $0<\alpha<\beta<+\infty$, it holds that
\begin{eqnarray}
t_{\max}&=&\max_{\alpha\leq x^TCx\leq\beta}x^TBx=\lambda_{\max}(B,C)\beta>0,\label{eq3}\\
t_{\min}&=&\min_{\alpha\leq x^TCx\leq\beta}x^TBx=\lambda_{\min}(B,C)\alpha>0.\label{eq4}
\end{eqnarray}
Thus $D=(\ref{eqd})$. 

$L$ denotes the L-smooth constant corresponding to $f(s,t)$. It holds that 
\[
\nabla f(s,t)=\left(\begin{array}{c}
1\\
-\frac{1}{2\sqrt{t}}
\end{array}\right),~~~~
\nabla^2 f(s,t)=\left(\begin{array}{cc}
0&~~0\\
0&~~\frac{1}{4\sqrt{t^3}}
\end{array}\right).
\]
The minimum of $t$ is $t_{\min}$ as defined in (\ref{eq4}).
%, and 
%$
%t_{\min}=\min_{\alpha\leq x^TCx\leq \beta}x^TBx\geq \lambda_{\min}(B,C).
%$
Then 
\[
\|\nabla^2 f(s,t)\|= \left\vert\frac{1}{4\sqrt{t^3}}\right\vert
\leq\left\vert\frac{1}{4\sqrt{t_{\min}^3}}\right\vert
=\frac{1}{4(\lambda_{\min}(B,C)\alpha)^{3/2}}\triangleq L.
\]
According to Definition \ref{def1}, a twice-differentiable convex function with a uniformly bounded Hessian norm is L-smooth. Thus, $f(s,t)$ is $L$-smooth on $\Omega$ with constant $L$. 

Substituting the obtained constants $L$ and $D$ into Theorem \ref{fw1} yields (\ref{converge}). The proof is complete.
\end{proof}
%
%\begin{rem}
%The computational iteration is independent of the dimension $n$.
%\end{rem}
%\textcolor{blue}{
%According to the upper bound (\ref{converge}) in Theorem \ref{thm3}, we can conclude that in the following special cases, {\rm (QR)} may converge faster:} 
%
%(i)For fixed $D$, when $\alpha$ gets larger, then $L$ gets smaller and the upper bound (\ref{converge}) gets smaller; 
%
%(ii) When $\beta$ gets smaller, then $D$ gets smaller and the upper bound (\ref{converge}) gets smaller; 
%
%(iii) In {\rm (QR)}, when $\alpha\approx\beta$ and $B\approx C$, it holds that $\lambda_{\max}(B,C)\approx\lambda_{\min}(B,C)\approx1$. Then the smaller $\alpha$ (as well as $\beta$) is, the smaller (\ref{converge}) gets since 
%it holds that
%\begin{eqnarray*}
%\frac{2LD^2}{k+2}&\approx&\frac{2}{k+2}\times\frac{1}{4\alpha^{3/2}}\times (\lambda_{\max}(A,C)-\lambda_{\min}(A,C))^2\alpha^2\nonumber\\
%&=&\frac{\sqrt{\alpha}(\lambda_{\max}(A,C)-\lambda_{\min}(A,C))^2}{2(k+2)}.\label{ubs}
%\end{eqnarray*}

In Theorem \ref{thm3}, we have proved that the Algorithm \ref{algfw} converges with respect to the optimal value. In the following, we will further prove that the algorithm also converges with respect to the iterates. The expressions of $t_{\min}$, $t_{\max}$, $s_{\min}$, $s_{\max}$, $D$, $L$ that defined in (\ref{eql})-(\ref{eq4}) will be used throughout this paper.

\section{Iterate convergence of the Frank-Wolfe algorithm (Algorithm \ref{algfw}) for {\rm (HP)}}
In this section, we will prove that Algorithm \ref{algfw} converges with iterates. We first provide some lemmas that is essential for further analysis.
\begin{lemma}\label{lem1}
$(s^*,t^*)$ is the optimal solution of {\rm (HP)} if and only if it is the optimal solution of
\begin{eqnarray}
&\min\limits_{(s,t)\in \Omega}&d(s,t)= s-\frac{1}{2\sqrt{t^*}}t\label{BP}.
\end{eqnarray}
\end{lemma}
\begin{proof}
Since {\rm (HP)} is a convex programming problem, the first order optimality condition is both necessary and sufficient. Thus $(s^*,t^*)$ is optimal for {\rm (HP)} if and only if 
\begin{eqnarray*}
\nabla f(s^*,t^*)^T\left(\begin{array}{c}
s-s^*\\
t-t^*
\end{array}\right)=\left(1,~-\frac{1}{2\sqrt{t^*}}\right)\cdot\left(\begin{array}{c}
s-s^*\\
t-t^*
\end{array}\right)
\geq 0,~~\forall(s,t)\in\Omega,
\end{eqnarray*}
which is equivalent to the following result that 
\begin{eqnarray}
s-\frac{1}{2\sqrt{t^*}}t\geq s^*-\frac{1}{2\sqrt{t^*}}t^*,~~\forall(s,t)\in\Omega.\label{con21}
\end{eqnarray}
Note that  
\[
\nabla d(s,t)=\left(\begin{array}{c}
1\\
-\frac{1}{2\sqrt{t^*}}
\end{array}\right)
\]
and (\ref{con21}) is precisely the optimality condition for problem (\ref{BP}). Therefore, $(s^*,t^*)$ is optimal for {\rm (HP)} if and only if it is optimal for (\ref{BP}). The proof is complete.
%
%Thus it holds for $\forall (s,t)\in\Omega$ that
%\begin{eqnarray}
%\nabla d(s^*,t^*)^T\left(\begin{array}{c}
%s-s^*\\
%t-t^*
%\end{array}\right)\geq 0,\label{con2}
%\end{eqnarray}
%which implies that $(s^*, t^*)$ is optimal for (\ref{BP}). The proof is complete.
\end{proof}

%\begin{theorem}\label{val1}
%Let $(s_k,t_k)$ be the sequences generated by the Algorithm \ref{algfw}, and let $(s^*,t^*)$ be the optimal solution of {\rm (HP)}. It holds that 
%\begin{eqnarray}
%s_k-\sqrt{t_k}-(s^*-\sqrt{t^*})\leq\frac{2LD^2}{k+2}.\label{res1}
%\end{eqnarray}
%%where $L$ and $D$ are defined as in (\ref{eql}) and (\ref{eqd}), respectively.
%\end{theorem}
%\begin{proof}
%According to Theorem \ref{thm3}, the result holds naturally.
%\end{proof}

\begin{lemma}\label{sq1}
Let $(s_k,t_k)$ be the sequences generated by the Algorithm \ref{algfw}, and let $(s^*,t^*)$ be the optimal solution of {\rm (HP)}. It holds that 
\begin{eqnarray}
\left\vert\sqrt{t_k}-\sqrt{t^*}\right\vert\leq\sqrt{\frac{4LD^2\sqrt{t_{\max}}}{k+2}}.\label{res2}
\end{eqnarray}
%where $t_{\max}$ is defined in (\ref{eq3}), $L$ and $D$ are defined in (\ref{eql}) and (\ref{eqd}), respectively.
\end{lemma}
\begin{proof}
According to Lemma \ref{lem1}, 
\[
s^*-\frac{1}{2\sqrt{t^*}}t^*\leq s_k-\frac{1}{2\sqrt{t^*}}t_k
\]
holds. Thus, it holds that
\begin{eqnarray*}
0&\leq& s_k-s^*+\frac{1}{2\sqrt{t^*}}(t^*-t_k)\\
&\leq&\frac{1}{2\sqrt{t^*}}(t^*-t_k)+\sqrt{t_k}-\sqrt{t^*}+\frac{2LD^2}{k+2}\\
&=&-\frac{(\sqrt{t_k}-\sqrt{t^*})^2}{2\sqrt{t^*}}+\frac{2LD^2}{k+2},
\end{eqnarray*}
where the second inequality holds due to Theorem \ref{thm3}. As a result,
\[
(\sqrt{t_k}-\sqrt{t^*})^2\leq \frac{2LD^2}{k+2}\times 2\sqrt{t^*}\leq \frac{4LD^2}{k+2}\sqrt{t_{\max}}.
\]
Consequently, (\ref{res2}) holds. The proof is complete.
\end{proof}
Now we are ready to prove that the sequences $\{s_k\}$ and $\{t_k\}$ of Algorithm \ref{algfw} converges.
\begin{theorem}{\rm (Convergence of $\{s_k\}$)}
Let $(s_k,t_k)$ be the sequences generated by the Algorithm \ref{algfw}, and let $(s^*,t^*)$ be the optimal solution of {\rm (HP)}. It holds that 
\begin{eqnarray}
-\sqrt{\frac{4LD^2\sqrt{t_{\max}}}{k+2}}\leq s_k-s^*\leq \frac{2LD^2}{k+2}+\sqrt{\frac{4LD^2\sqrt{t_{\max}}}{k+2}}.\label{res5}
\end{eqnarray}
%where $t_{\max}$ is defined in (\ref{eq3}), $L$ and $D$ are defined in (\ref{eql}) and (\ref{eqd}), respectively.
\end{theorem}
\begin{proof}
According to Theorem \ref{thm3} and Lemma \ref{sq1}, it holds that 
\begin{eqnarray}
s_k-s^*\leq \frac{2LD^2}{k+2}+\sqrt{t_k}-\sqrt{t^*}\leq \frac{2LD^2}{k+2}+\sqrt{\frac{4LD^2\sqrt{t_{\max}}}{k+2}}.\label{res3}
\end{eqnarray}
Since $(s^*,t^*)=\arg\min\limits_{(s,t)\in\Omega}s-\sqrt{t}$, it holds that $s_k-\sqrt{t_k}\geq s^*-\sqrt{t^*}$. Hence 
\begin{eqnarray}
-\sqrt{\frac{4LD^2\sqrt{t_{\max}}}{k+2}}\leq\sqrt{t_k}-\sqrt{t^*}\leq s_k-s^*,\label{res4}
\end{eqnarray}
where the first inequality holds due to Lemma \ref{sq1}.
%Theorem \ref{thm3}
 Combining (\ref{res3}) and (\ref{res4}) yields (\ref{res5}). The proof is complete.
\end{proof}
\begin{theorem}{\rm (Convergence of $\{t_k\}$)}
Let $(s_k,t_k)$ be the sequences generated by the Algorithm \ref{algfw}, and let $(s^*,t^*)$ be the optimal solution of {\rm (HP)}. It holds that 
\begin{eqnarray}
\left\vert t_k-t^*\right\vert\leq4\sqrt{t_{\max}}\sqrt{\frac{LD^2\sqrt{t_{\max}}}{k+2}}.\label{res6}
\end{eqnarray}
\end{theorem}
\begin{proof}
Since 
\[
\left\vert t_k-t^*\right\vert=\left\vert (\sqrt{t_k}-\sqrt{t^*})(\sqrt{t_k}+\sqrt{t^*})\right\vert \leq 
\sqrt{\frac{4LD^2\sqrt{t_{\max}}}{k+2}}\times 2\sqrt{t_{\max}}=(\ref{res6}),
\]
where the inequality follows from Lemma \ref{sq1}. The proof is complete.
\end{proof}
\section{Iterative minimum generalized eigenpair (IMGE) algorithm for globally solving {\rm (QR)}}  
In this section, we first prove that the optimal solution $\hat{x}_k$ of the Frank-Wolfe subproblem ${(\text{GP}_k)}$ is the $O(1/\sqrt{k})$-approximate solution of {\rm (QR)} with the help of Lemmas \ref{lem1} and \ref{sq1}. Then, we provide the iterative minimum generalized eigenpair (IMGE) algorithm for globally solving {\rm (QR)}. 

\begin{lemma}{\rm ($\delta_k$-approximate solution of (\ref{BP}))}\label{lem8}
Let $(\hat{s}_k,\hat{t}_k)$ be the optimal solution of the Frank-Wolfe subproblem defined in (\ref{p4}). 
%} at the $k$-th iteration:
%\begin{eqnarray}
%\min\limits_{(s,t)\in\Omega} s-\frac{1}{2\sqrt{t_k}}t\label{BP2}
%\end{eqnarray} 
%and let $(\hat{s}_k,\hat{t}_k)$ be its optimal solution. 
Then $(\hat{s}_k,\hat{t}_k)$ is a $\delta_k$-approximate solution of (\ref{BP}) with
\begin{eqnarray}
\delta_k=\frac{t_{\max}}{t_{\min}}\sqrt{\frac{LD^2\sqrt{t_{\max}}}{k+2}}.\label{res7}
\end{eqnarray}
\end{lemma}
\begin{proof}
Let $(s^*,t^*)$ be the optimal solution of {\rm (HP)}. According to Lemma \ref{lem1}, $(s^*,t^*)$ is also optimal for (\ref{BP}). Thus
\begin{eqnarray}
\hat{s}_k-\frac{1}{2\sqrt{t^*}}\hat{t}_k\geq s^*-\frac{1}{2\sqrt{t^*}}t^*.\label{ieq3}
\end{eqnarray}
Since $(\hat{s}_k,\hat{t}_k)$ is optimal for (\ref{p4}), it holds that 
\begin{eqnarray}
s^*-\frac{1}{2\sqrt{t_k}}t^*\geq \hat{s}_k-\frac{1}{2\sqrt{t_k}}\hat{t}_k.\label{ieq4}
\end{eqnarray}
We further obtain that 
\begin{eqnarray}
0\leq \hat{s}_k-s^*+\frac{1}{2\sqrt{t^*}}(t^*-\hat{t}_k)&\leq&\frac{1}{2\sqrt{t_k}}(\hat{t}_k-t^*)+\frac{1}{2\sqrt{t^*}}(t^*-\hat{t}_k)\nonumber\\
&=&\frac{1}{2}(\hat{t}_k-t^*)(\frac{1}{\sqrt{t_k}}-\frac{1}{\sqrt{t^*}})\nonumber\\
&=&\frac{1}{2\sqrt{t_k}\sqrt{t^*}}(\hat{t}_k-t^*)(\sqrt{t^*}-\sqrt{t_k})\nonumber\\
&\leq&\frac{t_{\max}}{2t_{\min}}\sqrt{\frac{4LD^2\sqrt{t_{\max}}}{k+2}}=(\ref{res7}),\nonumber
\end{eqnarray}
where the first inequality holds due to (\ref{ieq3}), the second inequality holds due to (\ref{ieq4}), the last inequality holds due to (\ref{res2}). Thus, 
\[0\leq d(\hat{s}_k,\hat{t}_k)-d(s^*,t^*)\leq \delta_k.\] 
The proof is complete. 
\end{proof}
\begin{theorem}{\rm ($\delta_k$-approximate solution of {\rm (HP)})}\label{th8}
Let $(\hat{s}_k,\hat{t}_k)$ be the optimal solution of the Frank-Wolfe subproblem defined in (\ref{p4}). 
%If $(\hat{s}_k,\hat{t}_k)$ is a $\delta_k$-approximate solution of problem (\ref{BP}), 
Then $(\hat{s}_k,\hat{t}_k)$ is a $\delta_k$-approximate solution of {\rm (HP)} with $\delta_k$ defined in (\ref{res7}).
\end{theorem}
\begin{proof}
Since $(\hat{s}_k,\hat{t}_k)$ is optimal for (\ref{p4}). According to Lemma \ref{lem8}, $(\hat{s}_k,\hat{t}_k)$ is a $\delta_k$-approximate solution of problem (\ref{BP}) that 
\begin{eqnarray}
d(\hat{s}_k,\hat{t}_k)-d(s^*,t^*)\leq\delta_k \label{da}
\end{eqnarray}
with $\delta_k$ defined in (\ref{res7}), and $(s^*,t^*)$ being the optimal solution of (\ref{BP}). And according to Lemma \ref{lem1}, $(s^*,t^*)$ is also optimal for {\rm (HP)}. Since for any $(s,t)\in\Omega$, it holds that
\begin{equation}
f(s,t)-f(s^*,t^*)=d(s,t)-d(s^*,t^*)-\frac{(\sqrt{t}-\sqrt{t^*})^2}{2\sqrt{t^*}}.\label{fa}
\end{equation}
Thus, we obtain
\[
f(\hat{s}_k,\hat{t}_k)-f(s^*,t^*)\leq d(\hat{s}_k,\hat{t}_k)-d(s^*,t^*)\leq \delta_k,
\]
where the first inequality follows from (\ref{fa}), the second inequality follows from (\ref{da}). Hence $(\hat{s}_k,\hat{t}_k)$ is further a $\delta_k$-approximate solution of {\rm (HP)}.
\end{proof}
%
%\textcolor{red}{Assume that (\ref{BP}) has the unique optimal solution. Then $(\hat{s}_k,\hat{t}_k)$ converges to $(s^*,t^*)$.}
%
%\textcolor{red}{The IMGE algorithm converges sublinearly in both objective value and iterates. Specifically,
%\begin{eqnarray*}
%% \nonumber % Remove numbering (before each equation)
%  f(s_k,t_k)-f(s^*,t^*) &=& \mathcal{O}(\frac{1}{k})? \\
%  \|(s_k,t_k)-(s^*,t^*)\| &=& \mathcal{O}(\frac{1}{\sqrt{k}})?
%\end{eqnarray*}
%}
\begin{theorem}{\rm ($\delta_k$-approximate solution of {\rm (QR)})}\label{sol}
Let $\hat{x}_k$ be the optimal solution of ${(\text{GP}_k)}$, then it is a $\delta_k$-approximate solution of {\rm (QR)} with $\delta_k$ defined in (\ref{res7}). 
\end{theorem}
\begin{proof}
Since $\hat{x}_k$ is the optimal solution of ${(\text{GP}_k)}$, the corresponding $(\hat{s}_k,\hat{t}_k)=(\hat{x}_k^TA\hat{x}_k,\hat{x}_k^TB\hat{x}_k)$ is optimal for (\ref{p4}). 
%According to Theorem \ref{th8}, $(\hat{s}_k,\hat{t}_k)$ is also a $\delta_k$-approximate solution of problem (\ref{BP}) that 
%\begin{eqnarray}
%d(\hat{s}_k,\hat{t}_k)-d(s^*,t^*)\leq\delta_k \label{da}
%\end{eqnarray}
%with $\delta_k$ defined in (\ref{res7}), and $(s^*,t^*)$ being the optimal solution of (\ref{BP}). And according to Lemma \ref{lem1}, $(s^*,t^*)$ is also optimal for {\rm (HP)}. Since for any $(s,t)\in\Omega$, it holds that
%\begin{equation}
%f(s,t)-f(s^*,t^*)=d(s,t)-d(s^*,t^*)-\frac{(\sqrt{t}-\sqrt{t^*})^2}{2\sqrt{t^*}}.\label{fa}
%\end{equation}
%Thus, we obtain
%\[
%f(\hat{s}_k,\hat{t}_k)-f(s^*,t^*)\leq d(\hat{s}_k,\hat{t}_k)-d(s^*,t^*)\leq \delta_k,
%\]
%where the first inequality follows from (\ref{fa}), the second inequality follows from (\ref{da}). 
Since $(\hat{s}_k,\hat{t}_k)$ is a $\delta_k$-approximate solution of {\rm (HP)} by Theorem \ref{th8}. It holds that
\begin{equation}
v{\rm (HP)}\leq f(\hat{s}_k,\hat{t}_k)\leq v{\rm (HP)}+\delta_k.\label{opt1}
\end{equation}
Insert the fact that 
$f(\hat{s}_k,\hat{t}_k)=\hat{x}_k^TA\hat{x}_k-\sqrt{\hat{x}_k^TB\hat{x}_k}=q(\hat{x}_k)$
and $v{\rm (HP)}=v{\rm (QR)}$ into (\ref{opt1}), we obtain
\begin{eqnarray*}
v{\rm (QR)}\leq q(\hat{x}_k)\leq v{\rm (QR)}+\delta_k,
\end{eqnarray*}
which indicates that $\hat{x}_k$ is a $\delta_k$-approximate solution of {\rm (QR)} with $\delta_k$ defined in (\ref{res7}). The proof is complete.
\end{proof}    

Theorem \ref{sol} indicates that $\hat{x}_k$ is a $\delta_k$-approximate solution of {\rm (QR)}, eliminating the need to solve the quadratic system (\ref{qua}). This leads to the following algorithm:

\begin{algorithm}[H]\caption{Iterative Minimum Generalized Eigenpair (IMGE) of $(M_k,C)$ for Globally Solving {\rm (QR)}.}\label{alg6}
    \textbf{Input:} $A\in\Bbb S^{n\times n}$, $B(\succ 0)\in\Bbb S^{n\times n}$, $C(\succ 0)\in\Bbb S^{n\times n}$, $\alpha\in \Bbb R,~\beta\in\Bbb R$ with $0<\alpha<\beta<\infty$. %Number of iterations $K$. 
    %Termination parameter $\Delta$.
    \\
    \textbf{Initialization:} $k=1$, $x_0\in\{x\in\Bbb R^n:\alpha\leq x^TCx\leq \beta\}$, $t_1=x_0^TBx_0$.
    %, $f_0=x_0^T(A-\frac{1}{2\sqrt{t_0+\epsilon}}B)x_0$, $\gamma_0=1$, 
    \\
    \textbf{Output:} The last iteration solution $\hat{x}_k$.
    \begin{algorithmic}[1]
        \While{termination criterion not satisfied}
        \State $M_k=A-\frac{1}{2\sqrt{t_k}}B$.
        \State Solve the optimal solution $\hat{x}_k$ of $(\text{GP}_k)$ via Proposition \ref{prop1}. %$(\hat{s}_k,\hat{t}_k)=(\hat{x}_k^TA\hat{x}_k,\hat{x}_k^TB\hat{x}_k)$ via Proposition \ref{prop1}.
        \State Update stepsize $\gamma_k$.
        \State 
        Update $t_{k+1}=(1-\gamma_k)t_k+\gamma_k\hat{t}_k$ with $\hat{t}_k=\hat{x}_k^TB\hat{x}_k$.
        \State Update $k:=k+1$.
        \EndWhile
    \end{algorithmic}        
\end{algorithm}
In each iteration of Algorithm \ref{alg6}, the main task is to solve $(\text{GP}_k)$, whose computational core lies in finding the minimum generalized eigenpair of the matrix pencil $(M_k,C)$. This motivates the name ``Iterative Minimum Generalized Eigenvalue (IMGE) Algorithm''. Moreover, only $t_k$ needs to be updated in each iteration.

\section{Numerical Experiments}
In this section, we compare the computational efficiency for solving {\rm (QR)} between the IMGE algorithm (Algorithm \ref{alg6}) with different stepsizes and {\rm (LC)} with CVX. All the numerical tests are implemented in Matlab R2020b and run on a laptop with Intel(R) Core(TM) i7-10710U CPU @ 1.10GHz 1.61 GHz processor and 16 GB RAM. 
\subsection{Some notations for the IMGE algorithm}
In this subsection, we state some notations for the IMGE algorithm:

(1) Initialization of $x_0$.
%$\hat{x}_0\in\{x\in\Bbb R^n:\alpha\leq x^TCx\leq \beta\}$. 
Choose an eigenvalue $\lambda$ of the matrix $C$ satisfying $\alpha<\lambda<\beta$ and set $x_0$ as the corresponding unit eigenvector. Then $\alpha\leq x_0^TCx_0\leq \beta$ holds automatically.

(2) Stepsize. We will mainly consider the following two types of stepsize in each iteration. 

(i) Diminishing stepsize. $\gamma_k^{dim}=\frac{2}{k+2}$. 

(ii) Exact line search ($\gamma_k^{exa}=\arg\min\limits_{\gamma\in[0,1]}\phi(\gamma)$). Consider the following one-dimensional programming problem:
\begin{eqnarray*}
\min\limits_{\gamma\in[0,1]} \phi(\gamma)&=&s_k+\gamma(\hat{s}_k-s_k)-\sqrt{t_k+\gamma(\hat{t}_k-t_k)}\\
&=&(1-\gamma)s_k+\gamma\hat{s}_k-\sqrt{(1-\gamma)t_k+\gamma \hat{t}_k},
\end{eqnarray*}
%
%\begin{eqnarray*}
%\arg\min\limits_{\gamma\in[0,1]} \phi(\gamma)=(1-\gamma)s_k+\gamma y_k-\sqrt{(1-\gamma)t_k+\gamma z_k},
%\end{eqnarray*}
where $\phi(\gamma)$ is convex with respect to $\gamma$. And 
\[
\phi'(\gamma)=\hat{s}_k-s_k-\frac{\hat{t}_k-t_k}{2\sqrt{t_k+(\hat{t}_k-t_k)\gamma}}.
\]
Then according to $\phi'(\gamma)=0$, we have $\gamma=\frac{\hat{t}_k-t_k}{4(\hat{s}_k-s_k)^2}-\frac{t_k}{\hat{t}_k-t_k}$. Thus, if $\hat{s}_k-s_k>0$ and $\hat{t}_k-t_k<0$, then $\gamma_k^{exa}=0$; if $\hat{s}_k-s_k<0$ and $\hat{t}_k-t_k>0$, then $\gamma_k^{exa}=1$; otherwise, $\gamma_k^{exa}=\gamma$.

(3) Stopping criterion. We adopt the Frank-Wolfe gap \cite{P2024} as the stopping criterion that we stop the algorithm when  
\[-\nabla f(s_k,t_k)^T\left(\begin{array}{c}
\hat{s}_k-s_k\\
\hat{t}_k-t_k
\end{array}\right)=s_k-\hat{s}_k+\frac{\hat{t}_k}{2\sqrt{t_k}}-\frac{\sqrt{t_k}}{2}\leq\Delta
\] 
holds or it reaches the predefined maximum iteration. 

(4) Minimum generalized eigenvalue decomposition. There are several methods for solving the generalized eigenvalue and the associated normalized generalized eigenvector, including the gradient method, the power method, the Lanczos method, and the preconditioned mirror descent method \cite{LX2025siam}. In this paper, we apply \textbf{min(eig)} function built in Matlab.

\subsection{Compare the computational efficiency for solving {\rm (QR)} among 3 methods}
In this subsection, we compare the computational efficiency for solving {\rm (QR)} among the following three methods: 

(i) Solve {\rm (QR)} through {\rm (LC)} with CVX. We use the default solver termination criteria (relative duality gap $\leq 10^{-6}$). We name this method by $M_{CVX}$; 

(ii) Apply the IMGE Algorithm with diminishing stepsize $\gamma^{dim}_k=\frac{2}{k+2}$. We name this method by $M_{dim}$;

(iii) Apply the IMGE algorithm with the stepsize obtained through exact line search $\gamma^{exa}_k$. We name this method by $M_{exa}$. 

We randomly generate the instances $A\in\Bbb S^{n\times n}$, $B\succ 0$, $C\succ 0$ with normal distribution and set $\alpha=1$, $\beta=10$. We denote the computational time (in seconds) of $M_{CVX}$, $M_{dim}$, $M_{exa}$ by $t_{CVX}$, $t_{dim}$, $t_{exa}$, respectively. Denote the iterations of $M_{dim}$ and $M_{exa}$ by $I_{dim}$, $I_{exa}$, respectively. Denote the optimal value obtained through $M_{CVX}$, $M_{dim}$, $M_{exa}$ by $v_{CVX}$, $v_{dim}$, $v_{exa}$, respectively. We use $\Delta$ and $Ite_{\max}$ to denote the termination tolerance and the predefined maximum iteration for the IMGE algorithm, respectively. We use $gap_{dim}$ and $gap_{exa}$ to denote the Frank Wolfe gap of $M_{dim}$ and $M_{exa}$, respectively. 
%All the computational time and iterations are reported with the average of 10 instances.

The computational time (in seconds), iterations, and the Frank-Wolfe gap for different dimensions ($N$) that varies from 100 to 700 are reported in Table \ref{time}. We set $\Delta=10^{-6}$ and $Ite_{\max}=2000$. It can be observed that for different dimensions, solving {\rm (QR)} through the IMGE algorithm with exact line search (i.e., $M_{exa}$) always outperforms the other two methods. 
%Additionally, for $N=700$, the reported average CPU time $t_{CVX}=1184.48$ seconds represents the average of 4 successful runs out of 5 attempts. The remaining 1 instance failed to produce a valid result. 
Additionally, for the fourth instance with $N=700$, the reported optimal value by CVX is $v_{CVX}=-363.7221$, which is significantly larger than $v_{exa}=-602.7299$. Moreover, the CVX solver reports non-convergence in the first two stages for this instance, which indicates that CVX solver fails to produce a valid result. Thus, the reported average CPU time $t_{CVX}=884.3284$ is computed over four successful runs out of five attempts.
%represents the average of 4 successful runs out of 5 attempts. 
%The remaining 1 instance failed to produce a valid result. 
Furthermore, when the dimension reaches 500, the IMGE algorithm with both diminishing and exact line search methods outperforms CVX solver. Thus, we only compare the computational time and iterations between methods $M_{exa}$ and $M_{dim}$ for the dimension above 700 in the following analysis. 
    \begin{table}[htbp]
        \centering
        \caption{Computational efficiency with $\Delta=10^{-6}$ and $Ite_{\max}=2000$.}        \label{time}
       % \small
        \setlength{\tabcolsep}{3pt}
        \begin{tabular}{c|cc|cccc|cccc}
             \hline
        $N=100$	&	$t_{CVX}$	&	$v_{CVX}$	&	$t_{dim}$	&	$I_{dim}$	&	$v_{dim}$	&	$gap_{dim}$	&	$t_{exa}$	&	$I_{exa}$	&	$v_{exa}$	&	$gap_{exa}$	\\
        \hline
1	&	4.88 	&	-229.2963	&	2.64 	&	2000	&	-229.2961	&	1.43e-04	&	0.02 	&	5	&	-229.2963	&	 5.21e-08	\\
2	&	2.18 	&	-222.0871	&	2.44 	&	2000	&	-222.0870	&	1.12e-04	&	0.03 	&	11	&	-222.0871	&	6.60e-08	\\
3	&	2.17 	&	-230.3269	&	2.23 	&	2000	&	-230.3268	&	1.32e-04	&	0.01 	&	5	&	-230.3270	&	3.81e-07	\\
4	&	1.84 	&	-133.1845	&	2.32 	&	2000	&	-133.1844	&	9.85e-05	&	0.01 	&	7	&	-133.1845	&	5.73e-07	\\
5	&	1.95 	&	-254.9110	&	2.34 	&	2000	&	-254.9109	&	1.13e-04	&	0.01 	&	4	&	-254.9110	&	2.09e-07	\\
avg	&	2.60 	&	-213.96116	&	2.39 	&	2000	&	-213.96104	&	1.2e-04	&	0.02 	&	6.4	&	-213.96118	&	2.57e-07	\\
\hline
N=300	&	$t_{CVX}$	&	$v_{CVX}$	&	$t_{dim}$	&	$I_{dim}$	&	$v_{dim}$	&	$gap_{dim}$	&	$t_{exa}$	&	$I_{exa}$	&	$v_{exa}$	&	$gap_{exa}$	\\
\hline
1	&	13.52 	&	-278.7148	&	17.57 	&	2000	&	-278.7147	&	 1.43e-04	&	0.06 	&	5	&	-278.7148	&	 2.56e-07	\\
2	&	13.37 	&	-486.7702	&	17.00 	&	2000	&	-486.7700	&	2.43e-04 	&	0.07 	&	5	&	-486.7702	&	  4.63e-09 	\\
3	&	12.35 	&	-381.4721	&	16.83 	&	2000	&	-381.4720	&	1.90e-04	&	0.06 	&	5	&	-381.4721	&	1.87e-07	\\
4	&	13.67 	&	-397.9059	&	19.40 	&	2000	&	-397.9060	&	2.51e-04	&	0.09 	&	6	&	-397.9063	&	6.41e-09 	\\
5	&	11.85 	&	-357.8252	&	16.72 	&	2000	&	-357.8251	&	 1.70e-04	&	0.07 	&	5	&	-357.8252	&	 1.81e-07 	\\
Average	&	12.95 	&	-380.53764	&	17.51 	&	2000	&	-380.53756	&	 1.99e-04	&	0.07 	&	5.2	&	-380.53772	&	 1.27e-07 	\\
\hline
N=500	&	$t_{CVX}$	&	$v_{CVX}$	&	$t_{dim}$	&	$I_{dim}$	&	$v_{dim}$	&	$gap_{dim}$	&	$t_{exa}$	&	$I_{exa}$	&	$v_{exa}$	&	$gap_{exa}$	\\
\hline
1	&	106.61 	&	-469.0959	&	65.63 	&	2000	&	-469.0957	&	2.96e-04	&	0.22 	&	5	&	-469.0960	&	 1.00e-08	\\
2	&	96.95 	&	-493.2039	&	69.92 	&	2000	&	-493.2037	&	2.37e-04	&	0.27 	&	5	&	-493.2040	&	3.33e-09	\\
3	&	93.58 	&	-450.6011	&	52.13 	&	2000	&	-450.6008	&	2.47e-04	&	0.24 	&	6	&	-450.6011	&	4.91e-09	\\
4	&	94.34 	&	-628.0649	&	69.97 	&	2000	&	-628.0646	&	2.79e-04	&	0.17 	&	4	&	-628.0649	&	1.52e-07	\\
5	&	90.52 	&	-445.1109	&	55.80 	&	2000	&	-445.1106	&	2.46e-04	&	0.23 	&	6	&	-445.1109	&	5.51e-09 	\\
Average	&	96.40 	&	-497.21534	&	62.69 	&	2000	&	-497.21508	&	2.61e-04	&	0.23 	&	5.2	&	-497.21538	&	 3.52e-08	\\
\hline
N=700	&	$t_{CVX}$	&	$v_{CVX}$	&	$t_{dim}$	&	$I_{dim}$	&	$v_{dim}$	&	$gap_{dim}$	&	$t_{exa}$	&	$I_{exa}$	&	$v_{exa}$	&	$gap_{exa}$	\\
\hline
1	&	949.56	&	-547.4755	&	192.68	&	2000	&	-547.4753	&	2.93e-04	&	0.61	&	5	&	-547.4755	&	2.08e-07	\\
2	&	891.37	&	-545.4861	&	231.98	&	2000	&	-545.4858	&	3.38e-04	&	0.56	&	5	&	-545.4861	&	2.30e-08	\\
3	&	898.51	&	-519.5839	&	163.92	&	2000	&	-519.5837	&	2.68e-04	&	0.76	&	6	&	-519.5839	&	1.37e-08	\\
4	&	820.03	&	\textbf{-363.7221}	&	136.31	&	2000	&	\textbf{-602.7296}	&	3.00e-04	&	0.62	&	5	&	\textbf{-602.7299}	&	5.68e-09	\\
5	&	797.88	&	-522.1429	&	235.02	&	2000	&	-522.1426	&	2.98e-04	&	0.60	&	5	&	-522.1429	&	3.61e-07	\\
Average	&	\textbf{884.33}	&	-499.6821	&	191.98	&	2000	&	-547.4834	&	2.99e-04	&	0.63	&	5.2	&	-547.48366	&	1.22e-07	\\
        \hline
        \end{tabular}
    \end{table}
%\end{landscape}

We can observe from Table \ref{time} that with a high-precision tolerance (e.g., $\Delta = 10^{-6}$), the diminishing stepsize method consistently reaches its iteration limit ($Ite_{\max}$). To allow for a meaningful comparison, we set $\Delta = 10^{-3}$ and report the results in Table \ref{tab2} with $Ite_{max}=2000$. The values are averaged over five instances. In this case, the diminishing stepsize method $M_{dim}$ terminates before reaching $Ite_{\max}$. For the same dimension, the method $M_{exa}$ still requires less time and fewer iterations than $M_{dim}$. Moreover, the number of iterations does not exhibit a clear pattern as the dimension varies. 
\begin{table}[h]
\centering
\captionsetup{justification=centering}
\caption{\centering Computational efficiency with $\Delta=10^{-3}$ and $Ite_{\max}=2000$.}\label{tab2} %for different Delta?
\begin{tabular}{c|cccccc}
\hline
$N$	&	900	&	1100	&	1300	&	1500	&	1700	&	1900	\\
 \hline
$t_{dim}$	&	243.38	&	649.08	&	916.39	&	1201.79	&	1729.92	&	2243.68	\\
$t_{exa}$	&	1.37	&	2.74	&	3.27	&	4.07	&	6.77	&	8.28	\\
\hline
$I_{dim}$	&	1147.4	&	1144	&	1217.6	&	1245	&	1211.8	&	1272.6	\\
$I_{exa}$	&	5.0	&	4.0	&	4.4	&	4.0	&	4.0	&	4.4	\\
\hline
\end{tabular}
\end{table}

\subsection{Analysis of the convergence of $(s_k,t_k)$}
We illustrate the convergence behavior of $(s_k,t_k)$ with Example 1 ($N=200$, $\Delta=10^{-4}$). %to see the variation of the objective value of ${\rm (HP)}$ with the number of iterations. 
The method $M_{dim}$ reaches the predefined maximum iteration limit (2000), while $M_{exa}$ converges after 7 iterations. Figure \ref{figure11} shows the evolution of $f(s_k,t_k)$ over the first 15 iterations, clearly indicating faster convergence of $M_{exa}$. In contrast, the method $M_{dim}$ with diminishing stepsize $\gamma_k^{dim}$ decays slowly, leading to slower progress. 

The advantage of $M_{exa}$ stems from its exact line search. Although the initial stepsize is small, it quickly becomes 1 after a few iterations (see Figure \ref{stepsize11}). This behavior is consistent across multiple instances. 
We conduct 50 experiments of $M_{exa}$ with varying dimensions. Among them, there are 3, 19, 16, 5, 7 groups of instances converge in 3, 4, 5, 6, 7 iterations, respectively. In all these cases, the stepsizes of the last few iterations always tend to be 1, as illustrated in the box plots in Figure \ref{fig:five} (only the cases of 4 and 5 iterations are shown for brevity, the remaining cases exhibit a similar pattern). Thus, the method $M_{exa}$ converges significantly faster compared with $M_{dim}$.

%
%The reason is that the objective value of the method $M_{dim}$ becomes quite large after the first iteration. And the stepsize $\gamma_k^{dim}$ gets smaller with the increasing iterations. While although the stepsize $\gamma_{k}^{exa}$ of $M_{exa}$ is quite small at first, it becomes 1 after a few steps as shown in Figure \ref{stepsize11}. 
%This phenomenon is not accidental. 
%
%We can observe that although the problems may be solved with different iterations, the stepsizes of the last few iterations always tend to be 1. 

%\begin{figure}[H]
%\centering
%\includegraphics[width=0.5\linewidth]{figf1}
%\caption{Variation of $f(s_k,t_k)$ with iterations}\label{fig3}
%\end{figure}
%\begin{figure}[H]
%\centering
%\includegraphics[width=0.5\linewidth]{figf4}
%\caption{Variation of stepsizes with iterations}\label{figf4}
%\end{figure}

\begin{figure}[htbp]
    \centering
    \begin{minipage}[b]{0.48\textwidth}
        \centering
        \includegraphics[width=\textwidth]{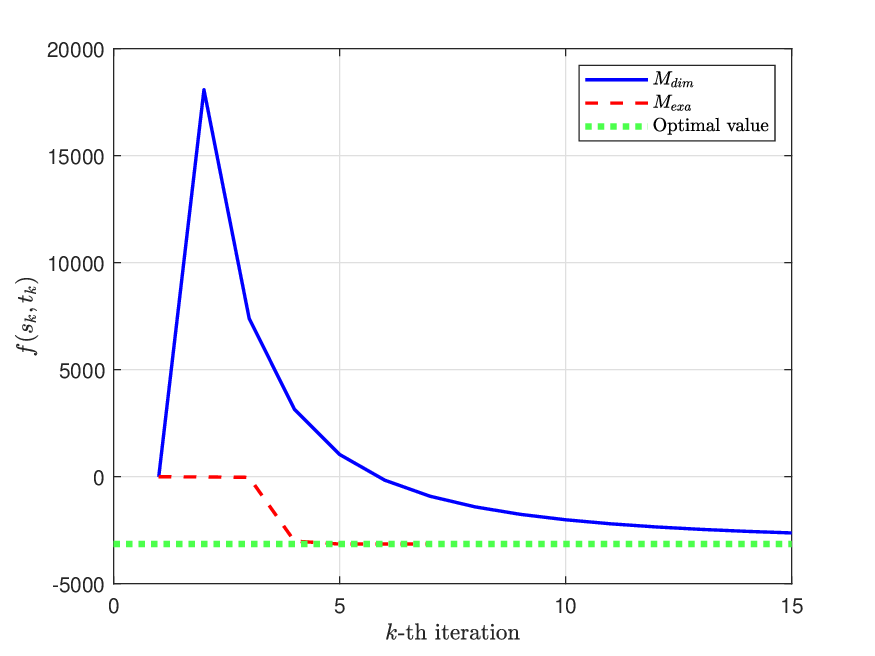}
        \captionof{figure}{Variation of $f(s_k,t_k)$}\label{figure11}
    \end{minipage}
    \hfill
    \begin{minipage}[b]{0.48\textwidth}
        \centering
        \includegraphics[width=\textwidth]{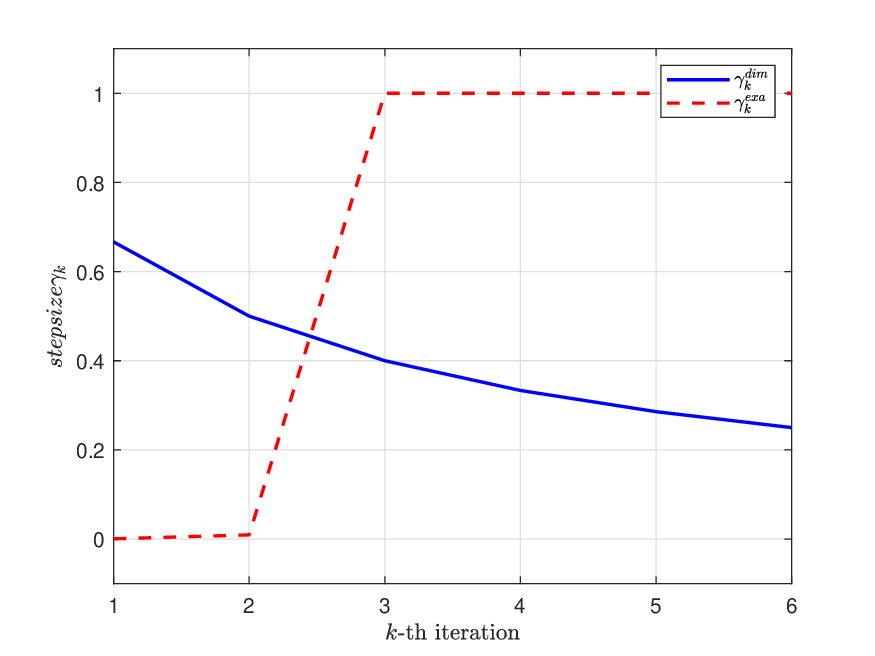}
        \captionof{figure}{Variation of stepsizes}\label{stepsize11}
    \end{minipage}\hfill
\end{figure}

\begin{figure}[htbp]
    \centering
   % \begin{minipage}[b]{0.5\textwidth}
%        \centering
%        \includegraphics[width=\textwidth]{ite3}
%        \captionof{subfigure}{3 iterations-3 groups of data}
%    \end{minipage}\hfill
    \begin{minipage}[b]{0.5\textwidth}
        \centering
        \includegraphics[width=\textwidth]{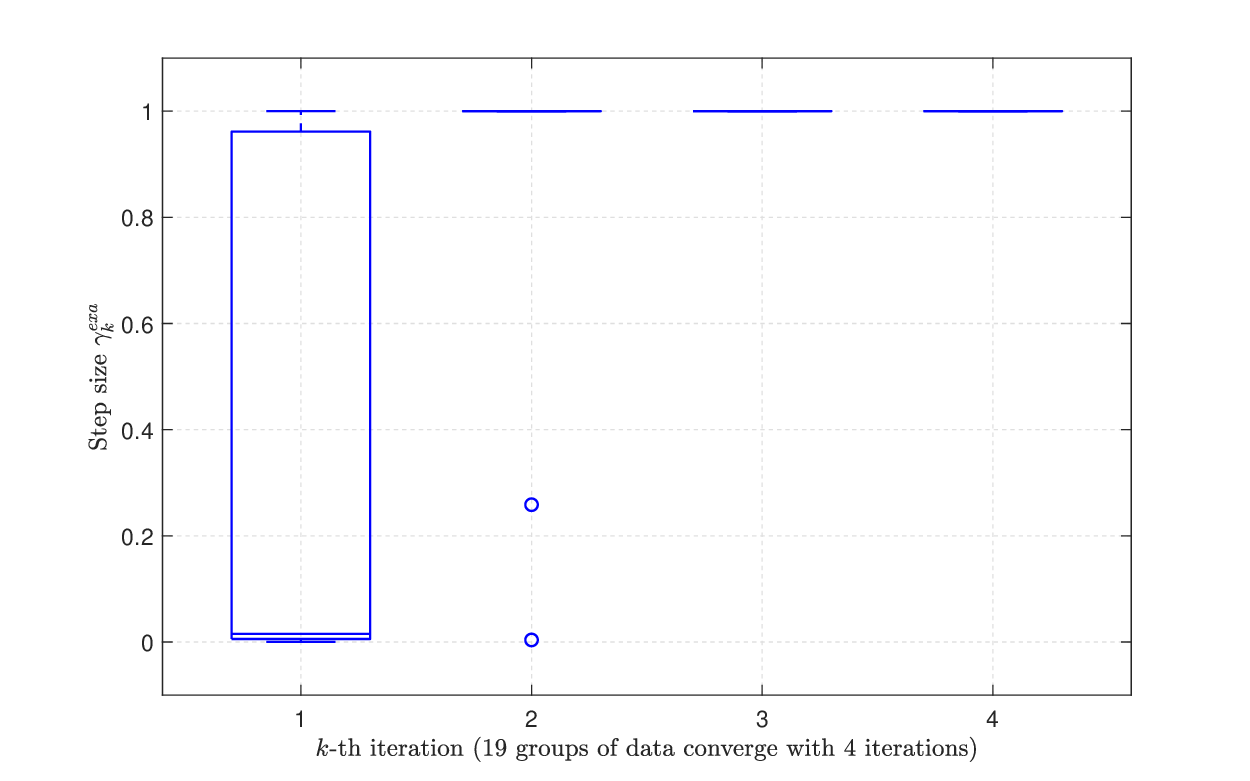}
        \captionof{subfigure}{4 iterations-19 groups of data}
    \end{minipage}\hfill
    \begin{minipage}[b]{0.5\textwidth}
        \centering
        \includegraphics[width=\textwidth]{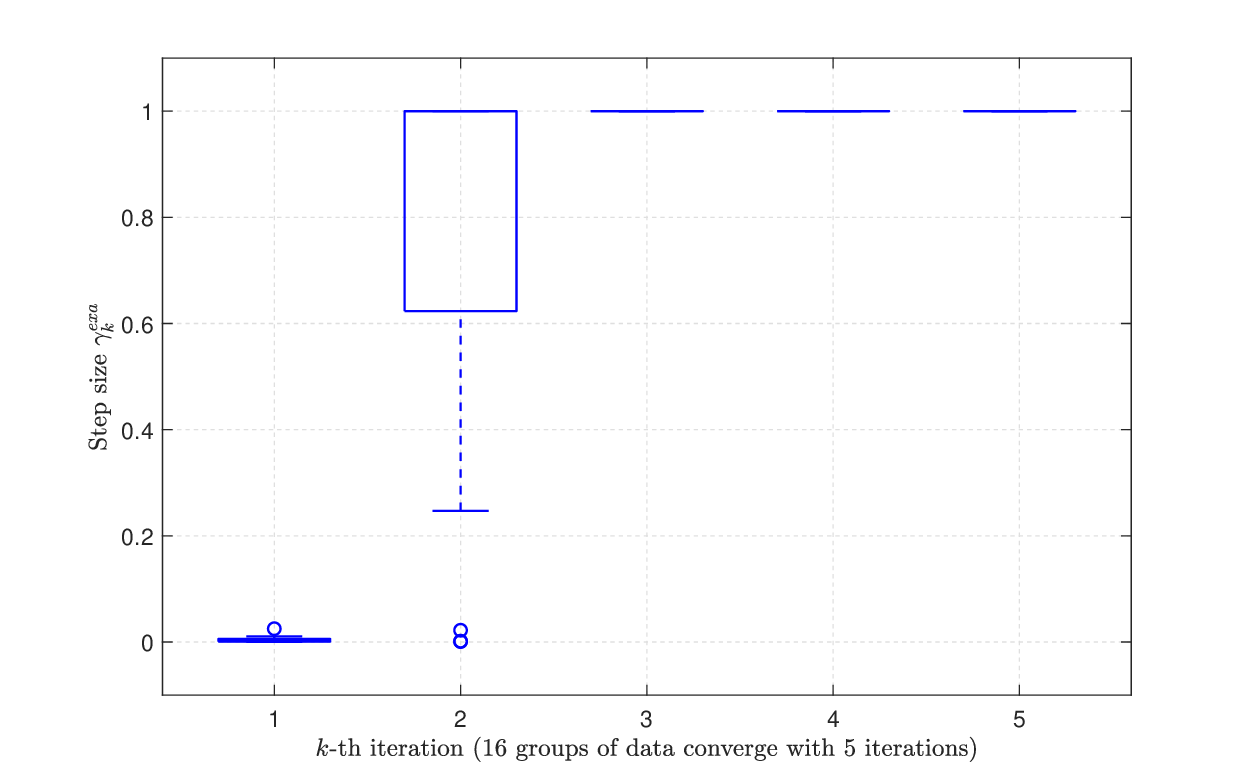}
        \captionof{subfigure}{5 iterations-16 groups of data}
    \end{minipage}\hfill
 %   \begin{minipage}[b]{0.5\textwidth}
%        \centering
%        \includegraphics[width=\textwidth]{ite6}
%        \captionof{subfigure}{6 iterations-5 groups of data}
%    \end{minipage}\hfill
%    \begin{minipage}[b]{0.5\textwidth}
%        \centering
%        \includegraphics[width=\textwidth]{ite7}
%        \captionof{subfigure}{7 iterations-7 groups of data}
%    \end{minipage}
    \caption{Variation of the stepsizes in each iteration}
    \label{fig:five}
\end{figure}
\subsection{Analysis of the convergence of $q(\hat{x}_k)$ (and $f(\hat{s}_k,\hat{t}_k)$)} %in objective value
In this subsection, we apply Example 1 to illustrate the convergence of $f(\hat{s}_k,\hat{t}_k)$ and $q(\hat{x}_k)$, which are the objective value of problems {\rm (HP)} and {\rm (QR)}, respectively. Since $f(\hat{s}_k,\hat{t}_k)=q(\hat{x}_k)$, we only compare $q(\hat{x}_k)$ with different methods $M_{dim}$ and $M_{exa}$. 
% We can see from Figure \ref{figf2} that $\hat{x}_k$ converges to the optimal value of {\rm (QR)} with both methods. 
The results show that $\hat{x}_k$ converges to the optimal value of {\rm (QR)} with both methods, which coincide with the theory of Theorem \ref{th8} and Theorem \ref{sol}. 
%Besides, it can be seen from Figure \ref{figf3} that the method $M_{exa}$ converges obviously faster. 
To highlight the trend, we scale the function values by $q_s(\hat{x}_k)=q(\hat{x}_k)*10^4$. As shown in Figure \ref{figf3}, although $q_s(\hat{x}_k)$ of $M_{exa}$ is initially larger (see Figure \ref{figf3} (left)), it turns much smaller from the fourth iteration and converges rapidly (see Figure \ref{figf3} (right)). The acceleration is due to the fact that the stepsize of $M_{exa}$ always tends to be 1 after a few steps (see Figure \ref{fig:five}). Besides, the update rule is $t_{k+1}=(1-\gamma_k^{exa})t_k+\gamma_k^{exa}\hat{t}_k$, which tends to be $t_{k+1}=\hat{t}_k$ when $\gamma_k^{exa}=1$. Thus, the algorithm is speeded up significantly. 
%\textcolor{red}{BUT dim stepsize is not always 1, how to explain its convergence with $\hat{x}_k???$ And when k gets quite large, $\gamma_k^{dim}$ is even close to 0???, then $t_{k+1}=t_k$???}
%
% \textcolor{red}{Figure \ref{fig:five} can also be used to explain the convergence of $(\hat{s}_k,\hat{t}_k)$.} 
%\begin{figure}[H]
%\centering
%\includegraphics[width=0.9\linewidth]{figf2}
%\caption{Variation of $q(\hat{x}_k)$ with the number of iterations}\label{figf2}
%\end{figure}

\begin{figure}[H]
\centering
\includegraphics[width=0.9\linewidth]{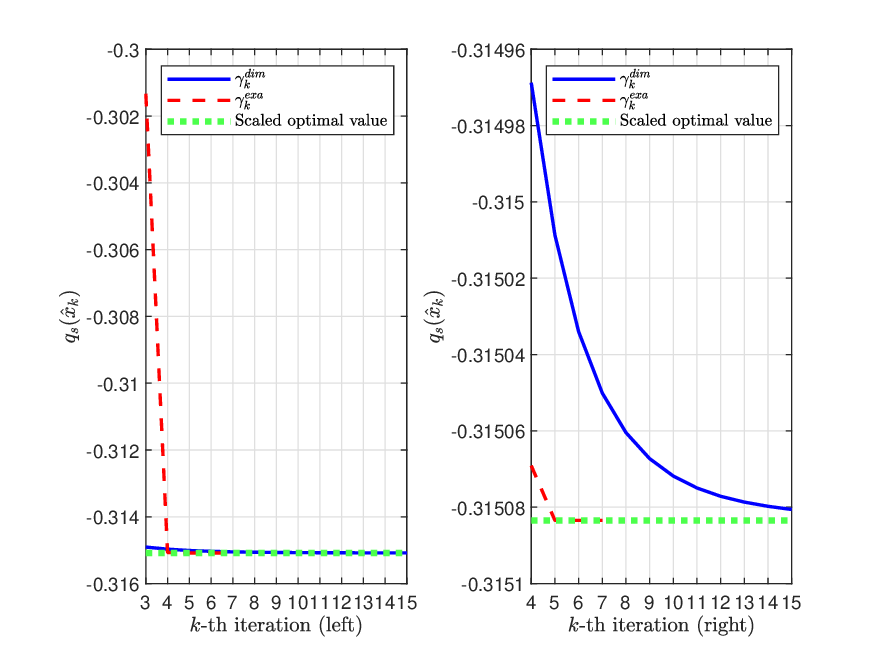}
\caption{Variation of $q_s(\hat{x}_k)$ with the number of iterations}\label{figf3}
\end{figure}

\section{Conclusion}
In this paper, we studied the nonconvex quadratic root-difference minimization problems under elliptic annulus constraints {\rm (QR)}, which is equivalent to a linear conic programming problem {\rm (LC)}. Thus {\rm (QR)} enjoys the hidden convexity property and can be solved through CVX solver. 
We equivalently reformulated {\rm (QR)} as a 2-dimensional convex programming problem {\rm (HP)} (different from {\rm (LC)}) with the help of newly developed Annulus Brickman theorem. However, it is challenging to solve {\rm (HP)} directly. Then we applied the Frank-Wolfe method on the hidden variables to solve {\rm (HP)}. The method converges with respect to not only the objective value, but also the iterates. We further prove that the solutions of the Frank-Wolfe subproblems are $O(1/\sqrt{k})$-approximate solutions of the original problem {\rm (QR)}, which eliminates the need to solve a computationally expensive quadratic system for solution recovery. With this recovery-free property, we develop the Iterative Minimum Generalized Eigenpair (IMGE) algorithm for globally solving {\rm (QR)}. Finally, numerical results demonstrate the computational efficiency and the convergency of the proposed IMGE algorithm. Future work will include inexactly solving the subproblem to accelerate IMGE algorithm.

\begin{appendices}
\section*{Appendix A. The case when $\alpha=0$ in {\rm (QR)}.}
Consider the case when $\alpha=0$ in {\rm (QR)}, we denote it by
\begin{eqnarray*}
{\rm(QR_0)}~ &\min\limits_{x\in\Bbb R^n}& x^TAx-\sqrt{x^TBx}\\
~~&{\rm s.t.}&0 \leq x^TCx\leq \beta.
\end{eqnarray*} 
Then it holds that
\begin{eqnarray}
v{\rm (QR_0)}&\leq& \min_{0\leq x^TCx\leq \beta}\lambda_{\max}(A,C)x^TCx-\sqrt{\lambda_{\min(B,C)}}\sqrt{x^TCx}\label{ieq1}\\
&=&\left\{
\begin{aligned}
&-\frac{\lambda_{\min}(B,C)}{4\lambda_{\max}(A,C)},~~if~\lambda_{\max}(A,C)>0,~\frac{\sqrt{\lambda_{\min}(B,C)}}{2\lambda_{\max}(A,C)}\leq\sqrt{\beta},\\
&\lambda_{\max}(A,C)\beta-\sqrt{\lambda_{\min}(B,C)\beta},~~otherwise,\nonumber
\end{aligned}
\right.\\
&\triangleq& \bar{f}<0,\nonumber
\end{eqnarray}
where (\ref{ieq1}) holds since
 $\lambda_{\min}(A,C)\leq \frac{x^TAx}{x^TCx}\leq \lambda_{\max}(A,C)$ and $\lambda_{\min}(B,C)\leq \frac{x^TBx}{x^TCx}\leq \lambda_{\max}(B,C)$. 

\begin{prop}
 Let $x^*$ be any minimizer of ${\rm (QR_0)}$. Then it holds that
 
\begin{eqnarray}
\sqrt{x^{*T}Cx^*}\geq\bar{\alpha}>0,\label{ub1}
\end{eqnarray} 
 where 
 \begin{eqnarray*}
\bar{\alpha}=\left\{
\begin{aligned}
&-\frac{\bar{f}}{\sqrt{\lambda_{\max}(B,C)}},~~if~\lambda_{\min}(A,C)\geq0,\\
&\frac{\sqrt{\lambda_{\max}(B,C)}-\sqrt{\lambda_{\max}(B,C)+4\lambda_{\min}(A,C)\bar{f}}}{2\lambda_{\min}(A,C)},~~if~\lambda_{\min}(A,C)<0. \label{ub}
\end{aligned}
\right.
 \end{eqnarray*}
%\begin{eqnarray}
%&&\sqrt{x^{*T}Cx^*}\nonumber\\
%&\geq&
%\left\{
%\begin{aligned}
%&-\frac{\bar{f}}{\sqrt{\lambda_{\max}(B,C)}},~~if~\lambda_{\min}(A,C)\geq0,\\
%&\frac{\sqrt{\lambda_{\max}(B,C)}-\sqrt{\lambda_{\max}(B,C)+4\lambda_{\min}(A,C)\bar{f}}}{2\lambda_{\min}(A,C)},~~if~\lambda_{\min}(A,C)<0 \label{ub}
%\end{aligned}
%\right.\\
%&\triangleq& \bar{\alpha}>0.\nonumber
%\end{eqnarray} 
\end{prop} 
\begin{proof} 
Since $x^*$ is any minimizer of ${\rm (QR_0)}$. Then it holds that
\begin{eqnarray}
0>\bar{f}\geq v{\rm (QR_0)}&=&x^{*T}Ax^*-\sqrt{x^{*T}Bx^*}\nonumber\\
&\geq& \lambda_{\min}(A,C)x^{*T}Cx^*-\sqrt{\lambda_{\max}(B,C)x^{*T}Cx^*}\label{ieq2}\\
&\geq&
\left\{
\begin{aligned}
&-\sqrt{\lambda_{\max}(B,C)x^{*T}Cx^*},~~if~\lambda_{\min}(A,C)\geq0,\\
&(\ref{ieq2}),~~otherwise.\nonumber
\end{aligned}
\right.
\end{eqnarray}
Then (\ref{ub1}) holds. The proof is complete.
\end{proof}

Following the above analysis, we obtain the lower bound of $x^TCx$ when $\alpha=0$. And we can solve ${\rm (QR_0)}$ through the following equivalent problem that
\begin{eqnarray*}
{\rm(QR_{\bar{\alpha}})}~ &\min\limits_{x\in\Bbb R^n}& x^TAx-\sqrt{x^TBx}\\
~~&{\rm s.t.}&\bar{\alpha} \leq x^TCx\leq \beta,
\end{eqnarray*} 
which has the similar formulation as {\rm (QR)} since $\bar{\alpha}>0$. Thus it can be solved through the IMGE algorithm that we proposed in this paper. 
This is the reason why we can assume that $\alpha>0$ without loss of generality in Assumption \ref{ass3}.

\section*{Appendix B. The relationship between {\rm (QR)} and the generalized eigenvector corresponding to $\lambda_{\max}(B,A)$.}
For $A,~B(\in\Bbb S^{n\times n})\succ 0$, the generalized eigenvector corresponding to $\lambda_{\max}(B,A)$ can be solved through the following Rayleigh quotient maximization problem:
\begin{eqnarray*}
{\rm (RQP)}~\max\limits_{x(\neq 0)\in\Bbb R^n} \frac{x^TBx}{x^TAx}.
\end{eqnarray*}
%corresponds to solving the largest generalized eigenvalue of $Bx=\lambda Ax$. And we denote the largest generalized eigenvalue by $\lambda_{\max}(B,A)$. 
 
In \cite{LX2025siam}, the authors proved that {\rm (RQP)} is equivalent to solve the following unconstrained minimization problem:
\begin{eqnarray*}
{\rm (UP)}~\min\limits_{x\in\Bbb R^n} x^TAx-\sqrt{x^TBx},
\end{eqnarray*}
where when $A=I$, it is to compute the eigenvector corresponding to $\lambda_{\max}(B)$.
This model was originally studied in \cite{A1989,A1991} and further researched in \cite{LXeig,LX2025siam}.

Denote the optimal solution of {\rm (UP)} by $x^*$. According to the first order optimality condition, it holds that
\[
Bx^*=2\sqrt{x^{*T}Bx^*}Ax^*.
\]
And $\lambda_{\max}(B,A)=2\sqrt{x^{*T}Bx^*}$ holds \cite{LX2025siam}. Besides, since
\[
\lambda_{\min}(B,C)\leq\frac{x^{*T}Bx^*}{x^{*T}Cx^*}\leq\lambda_{\max}(B,C).
\]
Thus
\[
\frac{\lambda_{\max}^2(B,A)}{4\lambda_{\max}(B,C)}\leq x^{*T}Cx^*\leq \frac{\lambda_{\max}^2(B,A)}{4\lambda_{\min}(B,C)}.
\]
As a result, if we set $\alpha\leq\frac{\lambda_{\max}^2(B,A)}{4\lambda_{\max}(B,C)}$, $\beta\geq\frac{\lambda_{\max}^2(B,A)}{4\lambda_{\min}(B,C)}$ and $A\succ 0$ in {\rm (QR)}, then it is equivalent to {\rm (UP)}.

\section*{Appendix C. The CDT viewpoint of {\rm (QR)}}
In {\rm (QR)}, we can introduce an auxiliary variable $y\in\Bbb R$ that $y^2\leq x^TBx$. 
Then it is equivalent to the following generalized nonhomogeneous CDT subproblem:
\begin{eqnarray*}
{\rm(GNCDT)}~ &\min\limits_{x\in\Bbb R^n,y\in \Bbb R}& x^TAx-y\\
~~&{\rm s.t.}&\alpha\leq x^TCx\leq \beta,\\
&&x^TBx\geq y^2.
\end{eqnarray*} 

It is well known that there is a duality gap for the general CDT subproblem \cite{CDT1985,SNTI2016,YZ2003}. 
However, there are still some special cases of CDT subproblem that enjoy the strong duality property. 
In \cite{SX2023}, the authors studied the homogeneous quadratic optimization in $\Bbb R^n~(n\geq 3)$ with two bilateral quadratic form constraints, 
which can also be seen as the homogeneous CDT subproblem with bilateral constraints {\rm (BHCDT)}. 
The authors proved that strong duality holds for {\rm (BHCDT)}. 
%While {\rm (GNHCDT)} can be seen as 
%%a homogeneous CDT subproblem in view of $x\in\Bbb R^n$,
%%but 
%a generalized nonhomogeneous CDT subproblem in view of the joint range of $(x,y)\in\Bbb R^n\times \Bbb R$.
In this part, we can conclude that {\rm (GNCDT)}, which is a generalized nonhomogeneous CDT subproblem, enjoys the hidden convexity property since it is equivalent to {\rm (QR)}.

Furthermore, let's review the following non-quadratic extension of homogeneous S-lemma proposed in \cite{YWX2023}, based on which we can equivalently reformulate {\rm (GNCDT)} as a convex programming problem.
\begin{theorem}
[\cite{YWX2023}]\label{Ulem:12}
Suppose $
f_i(x)=x^TA_ix,~i=0,1,...,m
$
 are quadratic forms on $\Bbb R^n$ $(n\geq 3)$ with $A_i=A_i^T,~\{A_i:i=0,...,m\}$ are linear combinations of $\{B_0,B_1,B_2\}\subset\mathcal{S}^n$, and
 $
 \omega_0B_0+\omega_1 B_1+\omega_2B_2\succ 0$
  for some scalars $\omega_0,\omega_1,\omega_2$. Let  $q_0(z),\ldots,q_m(z)$ be convex functions defined in a convex set $\bar{\Omega}\subseteq \Bbb R^n$. Assume there exist
 $\tilde{x}\in \Bbb R^n,~\tilde{z}\in{\rm relint} (\bar{\Omega})$ such that
$
 f_i(\tilde{x})+q_i(\tilde{z})<0,~i=1,\ldots,m.
$
 The following two statements are equivalent:
 \begin{itemize}
\item[{\rm (i)}] The following system is not solvable:
\[
f_0(x)+q_0(z)<0,~ f_i(x)+q_i(z)\leq 0,~i=1,\ldots,m,~z\in\bar{\Omega}.
\]
\item[{\rm (ii)}] There exist
$
\mu_i\geq 0,~i=1,2\ldots,m
$
 such that
\[
f_0(x)+\sum_{i=1}^m\mu_i f_i(x)+q_0(z) +\sum_{i=1}^m\mu_i q_i(z)\geq 0,~\forall~x\in \Bbb R^n,z\in \bar{\Omega}.
\]
\end{itemize}
\end{theorem}
Then it holds that
%
%\begin{eqnarray*}
%{\rm(GNHCDT)}~ &\min\limits_{x\in\Bbb R^n,y\in \Bbb R}& x^TAx-y\\
%~~&{\rm s.t.}&\alpha\leq x^TCx\leq \beta,\\
%&&x^TBx\geq y^2.
%\end{eqnarray*} 
%
\begin{eqnarray}
&&v({\rm GNCDT}) \nonumber\\
&=&\min~ \{x^TAx-y:~x^TCx\geq \alpha, x^TCx\leq \beta, -x^TBx+y^2\leq 0, \nonumber\\
&&~~~~~~~~~~~~~~~~~~~~~~~~~~~~~~~~~~~~~~~~~~~~~~~~~~~~~~~~~~~~~~~~~~~~~~~~~~\forall x\in\Bbb R^{n},\forall y\in\Bbb R\}\nonumber\\%&&~~~~~~~~~~~~~~~~~~~~~~~~~~~~~~~~~~~~~~~~~~~~~~\label{eq:2}\\
&=&\sup_{t}\{t:\{(x,y):~x^TAx-y<t, -x^TCx+\alpha \leq 0, x^TCx-\beta\leq 0, 
\nonumber\\
&&~~~~~~~~~~~~~~~~~~~~~~~~~~~~~~~~~~~~~~~~~~ -x^TBx+y^2\leq0,~\forall x\in\Bbb R^{n},\forall y\in\Bbb R\}=\emptyset\}\nonumber\\
&=&\sup_{t}\{t:\{~\exists\lambda\in\Bbb R^3_+:~x^TAx-y-t+\lambda_1(-x^TCx+\alpha)\nonumber\\
&&~~~~~~~~~~~+\lambda_2(x^TCx-\beta)+\lambda_3(-x^TBx+y^2)\geq0, \forall x\in\Bbb R^{n},\forall y\in\Bbb R\}\} \label{pb:1}\\
%&=&\sup_{\lambda\in\Lambda,t}\{t:~\inf_{x\in\Bbb R^n} x^T(A_0-\lambda_1A_1+\lambda_2A_1+\lambda_3A_2)x\nonumber\\
%&&~~~~~~~~~~~~~~~~~~~~~~~~~~~~~~~~~~~~~~~~~-t+\lambda_1m-\lambda_2M+\min_{y\in\Bbb R^m}\{ h(y)+\lambda_3g(y)\}\geq 0\} \nonumber\\
% %\label{pb:2}\\
&=&\sup_{t}\{t:\{~\exists\lambda\in\Bbb R^3_+:~x^T(A-\lambda_1C+\lambda_2C-\lambda_3B)x\nonumber\\
&&~~~~~~~~~~~~~~~~~~~~~~~~~~~~~~-t+\lambda_1\alpha-\lambda_2\beta+\lambda_3y^2-y\geq 0,
 \forall x\in\Bbb R^{n},\forall y\in\Bbb R\}\} \nonumber\\
&=&\sup_{\lambda\in\Bbb R^3_+,t} \left\{t: 
\inf_xx^T(A-\lambda_1C+\lambda_2C-\lambda_3B)x\right.\nonumber\\
&&\left.~~~~~~~~~~~~~~~~~~~~~~~~~~~~~~~~~~~~~~~~~~-t+\lambda_1\alpha-\lambda_2\beta+\inf_y\left\{\lambda_3y^2-y\right\}\geq0
\right\}\nonumber\\%\label{eq:1}\\
&=&\sup_{\lambda\in\Bbb R^3_+,t} \left\{t: 
A-\lambda_1C+\lambda_2C-\lambda_3B\succeq 0,~-t+\lambda_1\alpha-\lambda_2\beta-\frac{1}{4\lambda_3}\geq 0
\right\}\nonumber\\%\label{eq:1}\\
&=&\sup_{\lambda\in\Bbb R^3_+} \left\{\lambda_1\alpha-\lambda_2\beta-\frac{1}{4\lambda_3}: 
A-\lambda_1C+\lambda_2C-\lambda_3B\succeq 0
\right\}\nonumber\\%\label{eq:1}\\
&=&{\rm (LC)},\nonumber
%&=&\sup_{\lambda\in\Bbb R^3_+,\mu\geq0} \left\{\lambda_1\alpha-\lambda_2\beta-\mu: 
%A+(\lambda_2-\lambda_1)C-\lambda_3B\succeq 0,\left\| \left(\begin{array}{c}1\\
%\lambda_3-\mu \end{array}\right)\right\|\leq\lambda_3+\mu
%\right\},\nonumber%\label{eq:1}\\
\end{eqnarray}
where (\ref{pb:1}) holds due to Theorem \ref{Ulem:12} by letting 
$f_0(x)=x^TAx$, $q_0(y)=-y-t$, $f_1(x)=-x^TCx$, $q_1(y)=\alpha$, $f_2(x)=x^TCx$, $q_2(y)=-\beta$, $f_3(x)=-x^TBx$, $q_3(y)=y^2$.
%which is a convex programming problem and equivalent to {\rm (LC)}. 
The result indicates that we reveal the hidden convexity of {\rm (GNCDT)}. Besides, this result coincide with the previous conclusion that {\rm (GNCDT)} is equivalent to {\rm (QR)}.
\end{appendices}

\section*{Declarations}
\textbf{Funding}
This research was supported by the National Natural Science Foundation of China under grant 12101041, by Interdisciplinary Research Project for Young Teachers of USTB (Fundamental Research Funds for the Central Universities) under grant FRF-IDRY-24-026, by the Beijing Natural Science Foundation under grant 4252006.
\\
\\
\textbf{Conflict of interest}
The authors declared that they have no conflicts of interest to this work.
\\
\\
\textbf{Availability of data and materials}
 The data and code that support this study are available from the corresponding author upon request.


%\bibliography{ref}
%common bib file


\end{document}